\documentclass[12pt, leqno]{amsart}

\usepackage[OT2,T1]{fontenc}
\DeclareSymbolFont{cyrletters}{OT2}{wncyr}{m}{n}
\DeclareMathSymbol{\Sha}{\mathalpha}{cyrletters}{"58}

\usepackage{indentfirst}
\usepackage{amstext}
\usepackage{amsthm}
\usepackage{amsopn}
\usepackage{amsfonts}
\usepackage{amsmath}
\usepackage{latexsym}
\usepackage{amscd}
\usepackage{amssymb}
\usepackage{amsmath}
\usepackage[all,cmtip]{xy}
\usepackage{leftidx}
\usepackage{graphicx}
\usepackage{tikz}

=\oddsidemargin \font\teneufm=eufm10 \font\seveneufm=eufm7
\font\fiveeufm=eufm5
\newfam\eufmfam
\textfont\eufmfam=\teneufm \scriptfont\eufmfam=\seveneufm
\scriptscriptfont\eufmfam=\fiveeufm

\def\GG{\mathbb{G}}

\def\1{\mbox{\bf 1}}

\def\type{\mathrm{\bf type}}

 \DeclareMathOperator{\Hom}{Hom}
\DeclareMathOperator{\Aut}{Aut}
\DeclareMathOperator{\Autext}{Autext}
\DeclareMathOperator{\Int}{Int}
 
\DeclareMathOperator{\Isom}{Isom}
\DeclareMathOperator{\Isomext}{Isomext}
\DeclareMathOperator{\Isomint}{Isomint}
\DeclareMathOperator{\im}{im} 

\DeclareMathOperator{\End}{End} \DeclareMathOperator{\Id}{Id}

\DeclareMathOperator{\Ind}{Ind}

\DeclareMathOperator{\SU}{\rm SU}

\DeclareMathOperator{\GL}{\rm GL}
\DeclareMathOperator{\SL}{\rm SL}

\newtheorem{theorem}{Theorem}[subsection]

\newtheorem{corollary}[theorem]{Corollary}

\newtheorem{lemma}[theorem]{Lemma}

\newtheorem{proposition}[theorem]{Proposition}

\newtheorem{stheorem}{Theorem}[section]

\newtheorem{scorollary}[stheorem]{Corollary}
\newtheorem{slemma}[stheorem]{Lemma}
\newtheorem{sproposition}[stheorem]{Proposition}
\newtheorem{sremark}[stheorem]{Remark}

\newtheorem{sexamples}[stheorem]{Examples}

\theoremstyle{definition}
\newtheorem{remark}[theorem]{Remark}
\newtheorem{remarks}[theorem]{Remarks}

\numberwithin{equation}{section}

\def\ZZ{\mathbb{Z}}
\def\CC{\mathbb{C}}

\def\QQ{\mathbb{Q}}

\def\C{\mathbb C}
\def\R{\mathbb R}

\def\2int{\mathop{2\int}\nolimits}
\def\Dyn{\mathrm{\bf Dyn}}
\def\uDyn{\underline{\mathrm{Dyn}}}

\def\Transpt{\mathrm{Transpt}}
\def\uPsi{\underline{\Psi}}

\def\End{\mathop{\rm  End}\nolimits}

\def\Spec{\mathop{\rm Spec}\nolimits}

\def\Hom{\mathop{\rm Hom}\nolimits}

\def\Ind{\mathop{\rm Ind}\nolimits}

\def\Gal{\mathop{\rm Gal}\nolimits}

\def\Int{\mathop{\rm Int}\nolimits}

\def\Ind{\mathop{\rm Ind}\nolimits}
\def\Coind{\mathop{\rm Coind}\nolimits}
\def\Br{\mathop{\rm Br}\nolimits}
\def\Aut{\text{\rm{Aut}}}

\def\Int{\mathop{\rm Int}\nolimits}

\def\Isom{\mathop{\rm Isom}\nolimits}

\def\trd{\mathop{\rm trd}\nolimits}

\def\resp.{\mathop{\rm resp.}\nolimits}

\def\Ker{\mathop{\rm Ker}\nolimits}

\def\lgr{\longrightarrow}
\def\la{\longleftarrow}

\font\math=cmmi10
\def\varpi{\hbox{\math\char'44}}

\def\simlgr{\buildrel\sim\over\lgr}
\def\simla{\buildrel\sim\over\la}

\def\pa{\S\kern.15em }

\def\un{\uppercase\expandafter{\romannumeral 1}}
\def\deux{\uppercase\expandafter{\romannumeral 2}}
\def\trois{\uppercase\expandafter{\romannumeral 3}}
\def\quatre{\uppercase\expandafter{\romannumeral 4}}
\def\cinq{\uppercase\expandafter{\romannumeral 5}}
\def\six{\uppercase\expandafter{\romannumeral 6}}

\def\cskip{ \hskip -0.6 em}
\def\dskip{ \hskip -0.8 em}
\def\ccskip{ \hskip -1.2 em}

\def\hfl#1#2#3{\smash{\mathop{\hbox to#3{\rightarrowfill}}\limits
^{\scriptstyle#1}_{\scriptstyle#2}}}
\def\gfl#1#2#3{\smash{\mathop{\hbox to#3{\leftarrowfill}}\limits
^{\scriptstyle#1}_{\scriptstyle#2}}}

\begin{document}

\title[Maximal tori]{On maximal tori of algebraic groups of type $G_2$ }

\author{C. Beli}
\address{
   Institute of Mathematics Simion Stoilow of the Romanian Academy,
  Calea Grivitei 21,
 RO-010702 Bucharest, Romania.
}
\email{Beli.Constantin@imar.ro}

\author{P. Gille}
\address{
   Institute of Mathematics Simion Stoilow of the Romanian Academy,
  Calea Grivitei 21,
 RO-010702 Bucharest, Romania.
}
\thanks{C. Beli and P. Gille were supported by the Romanian IDEI project PCE$_{-}$2012-4-364 of the Ministry of National Education
CNCS-UEFISCIDI}
\email{pgille@imar.ro}

\author{T.-Y. Lee}
\address{Ecole Polytechnique F\'ed\'erale de Lausanne,
EPFL, Station 8,
CH-1015 Lausanne, Switzerland.
}
\email{ting-yu.lee@epfl.ch}
\date{\today}

\begin{abstract}
\noindent Given an octonion algebra  $C$ over a field $k$,
its automorphism group is an algebraic  semisimple $k$--group of type $G_2$.
We study the maximal tori of $G$ in terms of the algebra $C$.

\medskip


\medskip

\end{abstract}

\maketitle

\section{Introduction} \label{sec_intro}

For classical algebraic  groups, and in particular for arithmetic fields, the investigation
of maximal tori is an interesting topic in the theory of algebraic groups and
arithmetic groups, see the papers of Prasad-Rapinchuk \cite[\S 9]{PR1}, \cite{PR2}, and also
Garibaldi-Rapinchuk \cite{GR}.
It is also related to the Galois cohomology of quasi-split semisimple groups
by Steinberg's section theorem; that connection is an important ingredient of the paper.

Let $k$ be a field, let $k_s$ be a separable closure and denote by $\Gamma_k=\Gal(k_s/k)$
the absolute Galois group of $k$.
 In this paper, we study maximal tori of groups of type $G_2$.
We recall that a semisimple algebraic $k$-group $G$ of type $G_2$ is
the group of automorphisms of a unique octonion algebra $C$ \cite[33.24]{KMRT}.
We come now to the following invariant of maximal tori \cite{G2,R}.
Given a $k$--embedding of $i: T \to G$ of a rank two torus,  we have a natural action
of $\Gamma_k$ on the root system $\Phi\bigl(G_{k_s}, i(T_{k_s}) \bigr)$ and the yoga of twisted forms
defines then a cohomology class   $\type(T,i) \in H^1(k,W_0)$
which is called the type of the couple $(T,i)$.
Here  $W_0\cong \ZZ/2\ZZ \times S_3$ is the Weyl group of the Chevalley group of type $G_2$.
By Galois descent \cite[29.9]{KMRT}, a  $W_0$-torsor is nothing but a couple $(k',l)$ where  $k'$ (resp. \ccskip $l$) is
a quadratic (resp. \ccskip cubic) \'etale $k$--algebra.
The main problem is then the following: Given an octonion algebra $C$ and such a couple $(k',l)$,
under which additional  conditions is there a $k$--embedding $i: T \to G=\Aut(C)$ of type   $[(k',l)] \in H^1(k,W_0)$ ?

\smallskip

We give a precise answer when the cubic extension $l$ is not a field (\S \ref{subsec_biquad}).
When $l$ is a field, we use subgroups of type $A_2$ of $G$ to relate with maximal tori of
special unitary groups where we can apply results of Knus, Haile, Rost  and Tignol \cite{HKRT}.
This provides indeed a criterion which is quite complicated (\ref{prop_criterion}).

The problem above can be formulated in terms of existence of $k$--points
for a certain homogeneous space $X$ under $G$ associated to $k',l$, see \cite[\S 1]{Le} or \S \ref{subsec_var}.
We recall here Totaro's general question \cite[question 0.2]{To}.

\medskip

{\it  For  smooth connected affine $k$--group  $H$
over the  field $k$ and a homogeneous $G$-variety $Y$ such that $Y$
has a zero-cycle of degree $d >0$, does
$Y$ necessarily have a closed  \'etale point of degree dividing $d$?
}

\medskip

Starting with the Springer's odd extension theorem for quadratic forms,
there are several cases where the question has a positive answer, mainly for principal homogeneous spaces (i.e. torsors).
We quote here the results
by Totaro \cite[th. 5. 1]{To} and  Garibaldi-Hoffmann \cite{GH} for certain exceptional groups, Black \cite{B} for classical adjoint groups
and  Black-Parimala \cite{BP} for semisimple simply connected classical groups of rank $\leq 2$.

If the base field $k$ is  large enough  (e.g. $\QQ(t)$, $\QQ((t)))$, we can construct an homogeneous space $X$ under $G$ of the shape above
having a quadratic point and a cubic point but no $k$--point (Th. \ref{theo_cycle}). This provides a new class of counterexamples
to  the  question in the case $d=1$ which are geometrically speaking simpler than those of Florence \cite{Fl}
and Parimala \cite{Pa}.

\medskip

Finally, in case  of a number  field, we show that this kind of varieties
satisfies the Hasse principle. In this case, our results are effective, that is
we can describe the type of the maximal tori of  a given group of type $G_2$, for
example for the ``compact'' $G_2$ over the rational numbers (see Example \ref{ex_eff}).

\bigskip

Let us review the contents of the paper. In section \ref{sec_image}, we
recall the notion of type and oriented type for a $k$--embedding
 $i: T \to G$ of a maximal $k$--torus in a reductive $k$--group
$G$. We study then  the  image of the map $H^1(k,T) \to H^1(k,G)$ of Galois cohomology
and relate in the quasi-split case with Steinberg's theorem on Galois cohomology.
Section \ref{sec_octonions} gathers basic facts on octonion algebras which are used in
the core of the paper, namely  sections \ref{sec_embed} and \ref{sec_hermitian}.
The number field case is considered in the short section \ref{sec_hasse}.
Finally, Appendix \ref{sec_tori} deals with the Galois cohomology of of $k$-tori
and quasi-split reductive  $k$--groups over  Laurent series fields.

\medskip

 A. Fiori investigated independently maximal tori of algebraic groups of type $G_2$ \cite{Fi}
and their rational conjugacy classes. Though his scope is different,
certain tools are common with our paper, for example the definition and the study of the subgroup of type $A_2$ attached
to a maximal torus (proposition 5.5  in \cite{Fi}, \S 5.1 here).

\medskip

\noindent{\bf Acknowledgements.} We express our thanks to  Andrei Rapinchuk and to the referee for  valuable comments.
We thank Alexander  Merkurjev for pointing us some  application to Bruhat-Tits' theory (prop. \ref{prop_BT}).

\bigskip

\section{Maximal tori of reductive groups and image of the cohomology}\label{sec_image}

Let $G$ be a reductive $k$--group.
We are interested in maximal  subtori of $G$ and also in the images
of the map $H^1(k,T) \to H^1(k,G)$. We shall discuss refinements of the application ot
 Steinberg's theorem on rational conjugacy classes to Galois cohomology.

\subsection{Twisted root data}
\subsubsection{Definition} In the papers \cite[0.3.5]{Le} and \cite[6.1]{G3} in the spirit of  \cite{SGA3}, the notion of twisted root datas is defined over
an arbitrary base scheme $S$. We focus here on the case of the base field $k$ and
use the equivalence of categories between \'etale sheaves over $\Spec(k)$
and the category of Galois sets, namely sets equipped with a continuous  action of the absolute Galois group
$\Gamma_k$.

We recall from \cite[7.4]{Sp} that a root datum is a quadruple $\Psi=(M,R, M^\vee, R^\vee)$ where $M$ is a lattice, $M^\vee$ its dual,
$R \subset M$ a finite subset (the roots) , $R^\vee$ a finite subset of $M^\vee$ (the  coroots) and a bijection
$\alpha \mapsto \alpha^\vee$ of $R$ onto $R^\vee$ which satisfy the next axioms (RD1) and (RD2).

For each $\alpha \in R$, we define endomorphisms $s_\alpha$ of $M$ and
$s^\vee_\alpha$ of $M^\vee$ by
$$
s_\alpha(m) = m - \langle m, \alpha^\vee\rangle \,  \alpha ; \enskip
s^\vee_\alpha(f) = f - \langle  \alpha, f \rangle \,  \alpha^\vee \quad ( m \in M, f \in M^\vee).
$$
\noindent (RD1) For each $\alpha \in R$, $\langle \alpha, \alpha^\vee \rangle=2$;

\smallskip

\noindent (RD2) For each $\alpha \in R$, $s_\alpha(R)=R$ and  $s^\vee_\alpha(R^\vee)=R^\vee$.

\medskip

We denote by $W(\Psi) $ the subgroup of $\Aut(M)$ generated by
the $s_\alpha$, it is called the Weyl group of $\Psi$.

\subsubsection{Isomorphisms, orientation}
An isomorphism of root data $\Psi_1=(M_1,R_1, M_1^\vee, R_1^\vee) \simlgr \Psi_2=(M_2,R_2, {M_2}^\vee, {R_2}^\vee)$
consists in an isomorphism $f: M_1 \simlgr M_2$ such that $f$ (resp. \cskip $f^\vee$) induces
a bijection $R_1 \simlgr R_2$ (resp. \cskip ${R_2}^\vee \simlgr R_1^\vee$). Let $\Isom( \Psi_1, \Psi_2)$ be the scheme of isomorphisms between $\Psi_1$
and $\Psi_2$.
We define the quotient $\Isomext(\Psi_1, \Psi_2)$ by
$\Isomext(\Psi_1, \Psi_2)=  W(\Psi_2) \backslash \Isom( \Psi_1, \Psi_2)$ which is isomorphic to $\Isom( \Psi_1, \Psi_2)/ W(\Psi_1)$.

An orientation $u$ between $\Psi_1$ and $\Psi_2$ is an element $u \in \Isomext(\Psi_1, \Psi_2)$.
We can then define the set $\Isomint_u(\Psi_1, \Psi_2)$  of inner automorphisms with respect to the orientation
$u$ as the preimage of $u$ by the projection $\Isom( \Psi_1, \Psi_2) \to \Isomext( \Psi_1, \Psi_2)$.

We denote by $\Aut(\Psi)= \Isom( \Psi, \Psi)$ the group of automorphisms of the root datum $\Psi$ and we have an exact sequence
$$
1 \to W(\Psi)  \to \Aut(\Psi) \to \Autext(\Psi) \to 1
$$
where $\Autext(\Psi) = \Isomext(\Psi, \Psi)$ stands for the quotient group of automorphisms of $\Psi$
(called  the group of exterior or outer automorphisms of $\Psi$).
The choice of an ordering on the roots permits to define a set of positive roots $\Psi_+$,
its basis and  the Dynkin index $\mathrm{Dyn}(\Psi)$ of $\Psi$.
Furthermore, we have an isomorphism $\Aut(\Psi, \Psi_+) \simlgr \Autext(\Psi)$ so that the above sequence
is split.

\subsubsection{Twisted version}

 A twisted root datum is a root datum equipped with a continuous action of $\Gamma_k$.
 To distinguish with the absolute case, we shall use the notation $\underline \Psi$.
The Weyl group $W(\uPsi)$ is then a finite group equipped with an action of $\Gamma_k$.
If $\uPsi_1$, $\uPsi_2$ are two twisted root data,
the sets $\Isom(\uPsi_1, \uPsi_2)$, $\Isomext(\uPsi_1, \uPsi_2)$ are Galois sets.
An orientation between $\uPsi_1$, $\uPsi_2$ is an element $u \in \Isomext(\uPsi_1, \uPsi_2)(k)$
and the set $\Isomint_u(\uPsi_1, \uPsi_2)$ is then a Galois set.

\subsection{Type of a maximal torus}
We denote by $G_0$ the  split form of $G$. We denote by $T_0$ a maximal
$k$-split torus of $G_0$ and by $\Psi_0=\Psi(G_0,T_0)$ the associated
root datum.  We denote by $W_0$ the Weyl group of
$\Phi_0$ and by $\Aut(\Psi_0)$ its automorphism group.

Let $i: T \to G$ be a $k$--embedding as a maximal torus.
The root datum  $\uPsi(G, i(T))= \Psi(  G(T)_{k_s}, i(T)_{k_s})$ is
equipped with an action of the  absolute Galois group
$\Gamma_k$, so defines a twisted root datum.
It is a $k$--form of the constant root datum $\uPsi_0$  and we define
the type of $(T,i)$ as the isomorphism class of
$$
\bigl[ \uPsi(G, i(T)) \bigr] \in H^1(k, \Aut(\Psi_0)).
$$
Recall that  by Galois descent, those $k_s/k$--forms are
classified by the Galois cohomology pointed set $H^1( k, \Aut(\Psi_0))$.

If two embeddings $i$, $j$ have the same image, then $\type(T,i)=  \type(T,j) \in H^1(k, \Aut(\Psi_0))$.
If we compose $i: T \to G$  by an automorphism $f \in \Aut(G)(k)$,
we have  $\type(T,i)=  \type(T,f \circ i) \in H^1(k, \Aut(\Psi_0))$.

\begin{remark} {\rm  If $G$ is semisimple and  has no outer isomorphism
(and it is the case for groups of type $G_2$),
$W_0=\Aut(\Psi_0)$ and the next considerations will not add anything.
}
\end{remark}

We would like to have an invariant  with value in the Galois  cohomology of some Weyl group.
The strategy is to  ``rigidify'' by adding an extra data to $i: T \to G$, namely  an orientation with respect to a quasi-split form of $G$.

Given a $k$--embedding $i:T \to G$, we denote by $\uDyn(G,i(T))$ the Dynkin diagram
$k$--scheme of $\uPsi(G,i(T))$; it is  finite \'etale and then encoded in the Galois set $\Dyn(G_{k_s},i(T)_{k_s})$.
There is a canonical isomorphism: $\uDyn(G)\cong\uDyn(G,i(T))$. (\cite{SGA3} Exp. 24, 3.3.)

We denote by   $G'$  a quasi-split $k$--form of $G$. Let $(T', B')$ be a Killing couple of
$G'$ and denote by  $\uPsi'= \uPsi( G', T')$ the associated  twisted root datum,  by $W'=N_{G'}(T')/T'$
its Weyl group which is a twisted constant finite $k$--group.

Suppose that $G$ is semisimple simply connected or adjoint;
in this case the homomorphism $\Autext(G) \to \Aut_{Dyn}(\uDyn(G)) $ is an isomorphism \cite[XXIV.3.6]{SGA3}.
 We fix then an isomorphism $v: \uDyn(G') \simlgr \uDyn(G)$. Together with the canonical
 isomorphism $\uDyn(G)\cong\uDyn(G,i(T))$, it induces an isomorphism
 $\tilde v: \uDyn(G') \simlgr \uDyn(G,i(T))$. For $G$ semisimple simply connected or adjoint, the isomorphism $\tilde v$ defines equivalently an orientation
 $$
 u \in \Isomext\Bigl( \uPsi(G', T') , \uPsi(G, i(T)) \Bigr).
 $$
 Then the Galois set $\Isomint_u\Bigl( \uPsi(G', T') , \uPsi(G, i(T) )\Bigr)$ is
 a right $W'$--torsor and  its class in $H^1(k, W')$ is called the oriented type
of $i: T \to G$ with respect to the orientation $v$. It is denoted by $\type_v(T,i)$ and we bear in mind
that it depends of the choice of
$G'$ and of $v$.

\subsection{The quasi-split case}

We deal here with the quasi-split $k$--group $G'$ and
with the exact sequence $1 \to T' \to N_{G'}(T') \buildrel \pi \over  \to  W' \to 1$.
Here we have a canonical isomorphism $id: \uDyn(G') \cong  \uDyn(G')$ and then
a natural way to define an orientation for a $k$--embedding $j: E \to G'$ of a maximal
$k$--torus.
Keeping the notations above, let us state the following result.

\begin{theorem}\label{GKR} (Kottwitz)

(1) The map $\Ker\Bigl( H^1(k,N_{G'}(T')) \to H^1(k,G') \Bigr) \buildrel \pi_* \over \lgr
H^1(k,W')$ is onto.

\smallskip

(2) For each $\gamma \in H^1(k,W')$, there exists a $k$--embedding
$j:E \to G'$ of a maximal $k$--torus such that $\type_{id}\bigl( (E,j) \bigr)=\gamma$.
 \end{theorem}

In  Kottwitz's paper, the  result  occurs only as
 a result on embeddings of maximal tori \cite[cor. 2]{K}.
 It was rediscovered by Raghunathan \cite{R} and independently by the second author  \cite{G2}.
 The proof of (1) uses Steinberg's theorem on rational conjugacy classes and we can explain quickly how one can derive (2) from (1).
 Given $\gamma \in H^1(k,W')$, assertion (1) provides a principal homogenous space $P$ under $N'=N_{G'}(T')$
 together with a trivialization $\phi: G' \simlgr P \wedge^{N'} G'$ such that
 $\pi_*[P]= \gamma$. Then $\phi$ induces a trivialization at the level of twisted
 $k$-groups $\phi_*: G' \simlgr {^P{G'}}$. Now if we twist $i': T' \to G'$ by $P$, we get
 a $k$--embedding ${^Pi'}: {^PT'} \to    {^P{G'}} \buildrel \phi_*  \over \simla G'$
 and one checks that  $\type_{id}\bigl( {^PT'}, {^Pi'} \bigr)=\gamma$.

\subsection{Image of the cohomology of tori}

We give now a slightly more precise form of Steinberg's theorem \cite[th. 11.1]{St}, see also \cite[III.2.3]{Se}.

\begin{theorem} \label{theo_steinberg}  Let $[z]  \in H^1(k, G')$.
Let $i: T \to {_zG'}$ be a maximal $k$--torus of the twisted $k$--group ${_zG'}$.
Then there exists a $k$--embedding $j: T \to G'$ and $[a] \in H^1(k,T)$ such that
$j_*[a]= [z]$ and such that  $\type_{can}(T,i)= \type_{id}( T,j)$.
\end{theorem}

In the result, the first orientation is the canonical one, namely
arising from the  canonical  isomorphism
$\uDyn(G') \simlgr \uDyn\bigl(_z(G') \bigr)$.

\begin{proof}
If the base field is finite, there is nothing to do since $H^1(k,G')=1$ by Lang's theorem.
We can then assume than $k$ is infinite.
We denote by $P(z)$ the $G'$--homogeneous space defined by $z$ and by
$\phi: G'_{k_s} \simlgr P(z)_{k_s}$ a trivialization satisfying $z_\sigma = \phi^{-1} \circ \sigma(\phi)$ for each
$\sigma\in\Gamma_k$. It induces a trivialization $\varphi : G'_{k_s} \simlgr  ({_z(G')})_{k_s}$
satisfying $\mathrm{int}(z_\sigma) = \varphi^{-1} \circ \sigma(\varphi)$ for each
$\sigma\in\Gamma_k$.

We denote by $(G')^{sc}$ the simply connected cover of $DG'$ and by $f: (G')^{sc} \to G'$ the natural $k$--homomorphism.
Let $T^{sc}$ be ${(_zf)}^{-1}(i(T)).$
Let $g^{sc}$ be a regular element in $T^{sc}(k)$ and consider the ${G'}^{sc}(k_s)$--conjugacy class $\mathcal{C}$
of $ \varphi^{-1}(g^{sc})$ in $(G')^{sc}(k_s)$.
This conjugacy class is rational in the sense that it is stabilized by $\Gamma_k$
since ${\bigl( \varphi^{-1}(g^{sc}) \bigr)}= z_\sigma \,  {^\sigma\bigl(\varphi^{-1}(g^{sc}) \bigr)} \, z_\sigma^{-1} $
for each $\sigma \in \Gamma_k$.
According to Steinberg \cite[10.1]{St} (and \cite[8.6]{BoS} in the non perfect case),
${{\mathcal C} \cap (G')^{sc}(k)}$ is not empty, so that there exists $g_1^{sc} \in
(G')^{sc} (k)$ and $h^{sc} \in (G')^{sc}(k_s)$ such that $\varphi^{-1}(g^{sc}) =   (h^{sc})^{-1} \,    g_1^{sc} \, h^{sc}$.
We put
$g={_zf}(g^{sc})$, $g_1=f(g_1^{sc})$, $h= f(h^{sc})$,
 $T_1= Z_{G'}(g_1)$ and $i_1: T_1 \to   G'$.

Since $g \in ({_z(G')})(k)$ and $g_1 \in G'(k)$, we have $
h^{-1}\,    g_1 \, h =  z_\sigma \, \, {^\sigma(h^{-1} \,    g_1 \, h)} \, z_\sigma^{-1}= z_\sigma\,
 h^{-\sigma} \, g_1 \, {^\sigma h} \, z_\sigma^{-1}$
for each $s \in \Gamma_k$ whence
$$
g_1= a_\sigma \,  g_1 \,  a_\sigma^{-1}
$$
where $a_\sigma= h \, z_\sigma \, h^{-\sigma}$ is a $1$--cocycle cohomologous to
$z$ with values in $T_1(k_s)=Z_{G'}(g_1)(k_s)$.
It remains to show the equality on the oriented types.
By the rigidity trick (see proof of prop. 3.2 in \cite{G2}), up to replace $k$ by the function field
of the $T_1$--torsor defined by $a$, we can assume that $[a]= 1 \in H^1(k,T_1)$.
We write $a_\sigma= b^{-1}\, {^\sigma b}$ for some $b \in T_1(k_s)$ and we have
then $z_\sigma= (bh)^{-1} \, {^\sigma(bh)}$ and  $\varphi^{-1}(g) =   (bh)^{-1} \,    g_1 \, bh $.

Putting  $h_2= bh \in G'(k_s)$, we have then $z_\sigma= h_2^{-1} \, {^\sigma h_2}$ and  $\varphi^{-1}(g) =   h_2^{-1} \,    g_1 \, h_2$.
We get  $k$--isomorphisms $\phi_2= \phi \circ L_{h_2^{-1}}: G' \to P(z)$ and $\varphi_2= \varphi \circ \mathrm{int}(h_2^{-1}):
G' \simlgr {_z(G')}$, such that the following diagram commutes
$$
 \xymatrix{
 T_1 \ar[r]^{i_1} \ar[d]^{\varphi_2}_{\wr}  & {G'} \ar[d]^{\varphi_2}_{\wr} \\
 T \ar[r]^{i} & {_z(G')}  .
 }
 $$
 Thus $\type_{can}(T,i)= \type_{id}(T_1,i_1) \in H^1(k,W')$.
\end{proof}

\subsection{Image of the cohomology of tori, II}

We recall the following well-known fact.

\begin{lemma}\label{lem_conj_isom} Let $H$ be a reductive $k$-group and
 $T$ be a $k$--torus of the same rank as $H$.
Let  $i,j: T \to H$ be  $k$--embeddings of a
maximal $k$--torus $T$. If $j=\Int(h)\circ i$ for some $h\in H(k_s)$, then we have
$h^{-1}\leftidx^{\sigma}h\in i(T)(k_s)$ for all $\sigma$ in the absolute Galois group $\Gamma_k$.
\end{lemma}
\begin{proof}
For any $\sigma\in\Gamma$ and any $t\in T(k_s)$, we have $j(\leftidx^{\sigma}t)=
\leftidx^{\sigma}h\cdot i(\leftidx^{\sigma} t)\cdot\leftidx^{\sigma}h^{-1}$.
Therefore, we have $j=\Int(\leftidx^\sigma h)\circ i=\Int(h)\circ i$ and
$h^{-1} \, \leftidx^{\sigma}h$ is a $k_s$-point of the centralizer $C_{H}(i(T))=i(T)$
\end{proof}

\newpage

\begin{lemma}\label{lem_image} Let $H$ be a reductive $k$-group and
let $T$ be a $k$--torus of the same rank as $H$. Let $v$ be  an orientation of $H$ with respect to a quasi-split form $H'$.
Let  $i, j: T \to H$ be  $k$--embeddings of a
maximal $k$--torus $T$  which are $H(k_s)$--conjugate. Then we have
$\mathrm{Im}(i_*)=\mathrm{Im}(j_*) \subseteq H^1(k,H)$ and $\type_v(T,i)=\type_v(T,j)$.
\end{lemma}

\begin{proof}
Let $j=\Int(h)\circ i$ for some $h\in H(k_s)$. By Lemma~\ref{lem_conj_isom}, we have $h^{-1}\leftidx^{\sigma}h\in i(T)(k_s)$.
Let $[\alpha]\in\mathrm{Im}(j_*)$  and $\alpha$ be a cocycle with values in $j(T(k_s))$ which
represents $[\alpha]$. Define $\beta_\sigma=h^{-1}\alpha_\sigma\leftidx^{\sigma}h$.
Then  $\beta$ is cohomologous to $\alpha$ and $\beta_\sigma=(h^{-1}\alpha_\sigma h)\cdot
(h^{-1}\leftidx^{\sigma}h)\in i(T(k_s))$. Hence
$[\alpha]=[\beta]\in\mathrm{Im}(i_*)$, which shows that $\mathrm{Im}(i_*)=\mathrm{Im}(j_*) \subseteq H^1(k,H)$.

Let $T_1=i(T)$ and $T_2=j(T)$. Let $\Transpt_G(T_1,T_2)$ be the strict transporter from $T_1$ to $T_2$
(\cite{SGA3}, Exp. $\mathrm{VI}_{\mathrm{B}}$, Def. 6.1 (ii)).
Note that $\Transpt_G(T_1,T_2)$ is a right $N_G(T_1)$-torsor.
We have a canonical isomorphism

\begin{center}
$\Transpt_G(T_1,T_2)\wedge\Isomint_v(\uPsi',\uPsi(G,T_1)) \, \simlgr \,  \Isomint_v(\uPsi',\uPsi(G,T_2))$.
\end{center}

\noindent Since $j=\Int(h)\circ i$, we have $h\in\Transpt_G(T_1,T_2)(k_s)$ and $h$ defines a trivialization
$\phi_h: N_G(T_1) \to \Transpt_G(T_1, T_2)$ which sends
the neutral element to $h$. Let $W_1=N_G(T_1)/T_1$. Since $\phi_{h}^{-1}\circ \sigma(\phi_h)=h^{-1}\, \leftidx^{\sigma}h\in T_1(k_s)$,
the image of the class of $\Transpt_G(T_1,T_2)$ in $H^1(k,W_1)$ is trivial. Hence $\Isomint_v(\uPsi',\uPsi(G,T_1))\simeq \Isomint_v(\uPsi',\uPsi(G,T_2))$,
i.e. \cskip $\type_v(T,i)=\type_v(T,j)$.
\end{proof}

\begin{proposition} \label{prop_image} Let $T$ be a $k$-torus of the same rank as $G$. Let $i_1, i_2: T \to G$ be
$k$--embeddings of  $T$ in $G$. Let $v$ be  an orientation of $G$ with respect to a quasi-split form $G'$.
If  $\type_v(T,i_1)= \type_v(T,i_2) \in H^1(k,W')$, then there is a
$k$--embedding $j:T \to G$ such that $j(T)=i_1(T)$ and $j$, $i_2$ are $G(k_s)$-conjugate. In particular, the
images of $i_{1,*}$, $i_{2,*}$, $j: H^1(k,T ) \to H^1(k,G)$ coincide.
\end{proposition}

\begin{proof}

Let $T_1=i_1(T)$ and $T_2=i_2(T)$ and again put $W_i=N_G(T_i)/T_i$ for $i=1,2$.
Let $\eta$ denote the class of the $N_G(T_1)$--torsor $\Transpt_G(T_1,T_2)$ in $H^1(k, N_G(T_1))$ and $\overline{\eta}$ be
the image of $\eta$ in $H^1(k,W_1)$.
We have a canonical isomorphism $\Transpt_G(T_1,T_2)\wedge\Isomint_v(\uPsi',\uPsi(G,T_1))\simlgr \Isomint_v(\uPsi',\uPsi(G,T_2))$. Since
$\type_v(T,i_1)=\type_v(T,i_2)$, we have $\Isomint_v(\uPsi',\uPsi(G,T_1))\simeq\Isomint_v(\uPsi',\uPsi(G,T_2))$,
hence $\overline{\eta}$ is the trivial class in $H^1(k,W_1)$. Therefore
the $N_G(T_1)$-torsor $\Transpt_G(T_1,T_2)$ admits a reduction to $T_1$.
More precisely,  there exist a $T_1$--torsor $E_1$ and an isomorphism $E_1 \wedge^{T_1} N_G(T_1) \simlgr \Transpt_G(T_1,T_2)$ of $N_G(T_1)$--torsors.
We take a point $e_1 \in E_1(k_s)$ and consider  its image $g$ in $G(k_s)$ under the mapping
$E_1 \wedge^{T_1} N_G(T_1) \simlgr \Transpt_G(T_1,T_2) \hookrightarrow G$.
Then  $g^{-1}\, \leftidx^{\sigma}g$ is a $k_s$-point of the centralizer $C_{G}(T_1)=T_1$ for all $\sigma\in\Gamma_k$.
We define a $k$-embedding $j:T\to G$ as $j(t)=(\Int(g^{-1})\circ i_2)(t)$.
To see that $j$ is indeed defined over $k$, we check as follows:
\begin{equation*}
\begin{split}
j(\leftidx^{\sigma}t)&=(\Int(g^{-1})\circ i_2)(\leftidx^{\sigma}t)\\
&=\Int(g^{-1})(\leftidx^{\sigma}i_2(t))\\
&=h\cdot\leftidx^{\sigma}((\Int(g^{-1})\circ i_2)(t))\cdot h^{-1}\\
&=\leftidx^{\sigma}(j(t)).
\end{split}
\end{equation*}
 By our construction, we have $j(T)=i_1(T)$ and $i_2$, $j$ are conjugated.
Let $f=(j|_{T_1})^{-1}\circ i_1$. Then $f$ is an automorphism of $T$ and $i_1=j\circ f$. Hence the images of $i_{1,*}$ and $ j_*$ coincide.
By Lemma~\ref{lem_image}, the images of $j$ and $ i_{2,*}$ coincide.
\end{proof}


This applies to the quasi-split case and enables us to slightly refine Theorem \ref{theo_steinberg}.

\begin{corollary}\label{cor_steinberg}
With the notations of Theorem \ref{theo_steinberg},    choose (by Theorem \ref{GKR})
for each class $\gamma \in H^1(k,W')$
 a $k$--embedding
$i(\gamma): E(\gamma) \to G'$  of oriented type $\gamma$.
Then the map
$$
\bigsqcup\limits_{\gamma \in H^1(k,W')} H^1(k, E(\gamma))  \enskip
\xrightarrow{\quad \sqcup \, i(\gamma)_* \quad} \enskip
 H^1(k,G')
$$
 is onto.
\end{corollary}

\subsection{Varieties of embedding  $k$--tori}\label{subsec_var}

Let $T$ be a $k$-torus and  $\uPsi$ be a twisted root datum of $\Psi_0$ attached to $T$, i.e. \dskip the character group of $T$
 is isomorphic to the character group encoded in $\Psi$. In this section, we will define a $k$-variety $X$ such that the existence
 of a $k$-point of $X$ is equivalent to the existence of a $k$-embedding of $T$ into $G$ with respect to $\uPsi$.

We start with a functor. The \emph{embedding functor} $\mathcal{E}(G,\uPsi)$ is defined as follows:
for any $k$-algebra $C$,
\[ \mathcal{E}(G,\uPsi)(C)=\left\{\begin{array}{l}
\mbox{$f:T_{C}\hookrightarrow G_{C}$}\left|\begin{array}{l}\mbox{$
f$ is both a closed
immersion and a group }\\
 \mbox{homomorphism which induces an isomorphism
}\\
\mbox{$f^{\Psi}:\uPsi_{C}\xrightarrow{\sim}\uPsi(G_{C},f(T_{C}))$
such that }\\
\mbox{ $f^{\Psi}(\alpha)=\alpha\circ f^{-1}|_{f(T_{C'})}$ is in
$\uPsi(G_{C'},f(T_{C'}))$,}\\
\mbox{for all $C'$-roots $\alpha$, for all $C$-algebra
$C'$.}
\end{array}\right.\end{array}\right\}.\]
In fact, the functor $\mathcal{E}(\uPsi,G)$ is representable by a $k$-scheme (\cite{Le} Theorem 1.1).
Define the Galois set $\Isomext(\uPsi, G )$ by $\Isomext(\uPsi, G )= \Isomext(\uPsi, \uPsi(G,E) )$
where $E$ stands  for an arbitrary maximal $k$--torus of $G$.
Given  an orientation $v \in   \Isomext(\uPsi, G )(k)$,
we define the \emph{oriented
embedding functor} as follows: for any $k$-algebra $C$,
\[\mathcal{E}(G,\uPsi,v)(C)=\left\{\begin{array}{l}\mbox{$f:T_{C}\hookrightarrow G_{C}$}\left|
\begin{array}{l}\mbox{$f\in\mathcal{E}(G,\uPsi)(C)$,
and the image of $f^{\Psi}$ } \\
\mbox{ in ${\Isomext}(\uPsi,G)(C)$ is
$v$.}\end{array}\right.\end{array}\right\}.\]
\bigskip

We have the following result:

\begin{theorem} In the sense of the \'{e}tale
topology, $\mathcal{E}(G,\uPsi,v)$ is a left homogeneous space under the
adjoint action of $G$, and a torsor over the variety of the maximal tori of G under the
right $W(\uPsi)$-action. Moreover, $\mathcal{E}(G,\uPsi,v)$ is
representable by an affine $k$-scheme.
\end{theorem}
\begin{proof}
We refer to Theorem 1.6 of \cite{Le}.
\end{proof}

\begin{remark}{\rm The definition of varieties of embeddings  is quite abstract but is simplified a lot
if there is a $k$--embedding  $i:T \to G$ of oriented type isomorphic to $(\uPsi,v)$. Indeed in this case,
the $k$--variety $\mathcal{E}(G,\uPsi,v)$ is $G$--isomorphic to the homogeneous space $G/i(T)$ and we observe that the map
$G/i(T) \to G/N_G(i(T))$ is a $W_G(i(T))$--torsor over the variety of maximal tori of $G$.
 }
\end{remark}

\begin{remark} {\rm We sketch another way to prove Theorem \ref{theo_steinberg}.
With the notations of that result, let  $z \in Z^1(k,G')$ and put $G={_zG'}$.
Let $T $ be a maximal $k$--torus  of $G$ and consider the twisted root data
$\uPsi=\uPsi(G, T)$ attached to $T$. Let $v$ be the canonical element in ${\Isomext}(\uPsi,G)(k)$ and let $v'=c\circ v$,
where $c\in{\Isomext}(G,G')(k)$ corresponds to the canonical orientation $\uDyn(G)\cong\uDyn(G')$.
 We denote by $X$ (resp. \cskip $X'$) the $k$--variety of
oriented embeddings of $T$ in $G$ (resp. \cskip $G'$) with respect to $\uPsi$
and $v$ (resp. \cskip $v'$).
Note that $G'$ acts on $X'$ and we  have a natural isomorphism $X \simlgr {_zX'}$.
Theorem \ref{GKR}.(2) shows that $X'(k) \not = \emptyset$ and the choice of a $k$--point $x'$
of $X'$ defines a $G'$--equivariant isomorphism $G'/T \simlgr X'$. In the other hand,
the embedding $i$ defines  a $k$--point $x \in X(k)$.
Since $X \cong {_zX'}$, we have that  $ {_z(G'/T)}(k) \not = \emptyset$, hence the class
$[z] \in H^1(k,G)$ admits a reduction to
$i': T \hookrightarrow G'$  such that $\type_{can}(T,i)= \type_{id}(T,i') \in H^1(k,W')$.
}
\end{remark}

\bigskip

\section{Generalities on octonion algebras}\label{sec_octonions}

Let $C$ be an octonion algebra.
We denote by $G$ the automorphism group of $C$, it is a semisimple $k$--group of type
$G_2$.
We denote by $N_C$ the norm of $C$, it is $3$-fold Pfister form.
In particular, $N_C$ is hyperbolic (equivalently isotropic) iff
$G$ is split  (equivalently isotropic).

\subsection{Behaviour under field extensions}

If $l/k$ is a field extension of odd degree, the Springer odd extension theorem \cite[18.5]{EKM} implies that
$C$ is split if and only if $C_l$ is split.
More generally, we have the following criterion.

\begin{lemma}\label{lem_coprime}
Let $(k_j)_{j=1,...,n}$ be a  family of finite field extensions such that is $g.c.d.([k_j:k])$ is odd.
Then $C$ is split if and only if $C_{k_j}$ is split for $j=1,...,n$. \qed
\end{lemma}

\begin{proof}
The direct sense is obvious.
Conversely, assume that $C_{k_j}$ is split for $j=1,...,n$. Then there exists an index $j$
such that $[k_j:k]$ is odd, hence $C$ splits.
\end{proof}

\begin{remark}
 { This is a special case of the following  more general result by Garibaldi-Hoffmann \cite[th. 0.3]{GH}
 answering positively Totaro's question.
Let $(k_j)_{j=1,...,n}$ be a  family of finite field extensions  and put  $d=g.c.d.([k_j:k])$.
Let $C$, $C'$ be Cayley $k$--algebras such that $C_{k_j}$ and $C'_{k_j}$ are isomorphic  for $j=1,...,n$, then there exists
a separable finite field extension $K/k$ of
 degree dividing $d$ such that $C_K$ is isomorphic to $C'_K$. This is the case of groups of type $G_2$ in that theorem
 which includes also the case of certains groups of type $F_4$ and $E_6$.
 }
\end{remark}

We recall also the behaviour with respect to quadratic \'etale algebras.

\begin{lemma}\label{lem_quadratic}
Let $k'/k$ be a quadratic \'etale algebra. Then
the following are equivalent:

\smallskip

(i) $C \otimes_k k'$ splits;

\smallskip

(ii) There is an isometry $(k',n_{k'/k}) \to (C,N_C)$  where $n_{k'/k}:k' \to k$
stands for the norm map;

\smallskip

(iii) There exists a embedding of unital composition $k$--algebra $k' \to C$.

\end{lemma}

\begin{proof}
If $C$ is split, all three facts hold so that we can assume that $C$ is not split.

\smallskip

\noindent $(i) \Longrightarrow (ii):$
Since $C$ is not split, il follows that $k'$ is a field.
Since  $N_C$ is split over $k'$, there exists a non-trivial and  non-degenerate
symmetric bilinear form $B$ such that $B \otimes n_{k'/k}$ is a subform of $N_C$
 \cite[34.8]{EKM}.
Since $N_C$ is multiplicative,   there is an isometry $(k',n_{k'/k}) \to (C,N_C)$.

 \smallskip

 \noindent $(ii) \Longrightarrow (iii): $
 Since the orthogonal  group $O(N_C)(k)$ acts transitively on the sphere $\{ x \in  C \, \mid \, N_C(x)=1 \}$,
 we can assume that our isometry $(k',n_{k'/k}) \to (C,N_C)$ maps $1_{k'}$ to $1_C$.
 It is then a map of unital composition $k$--algebras.

 \smallskip

 \noindent $(iii) \Longrightarrow (i): $
 If $k'=k \times k$, then $N_C$ is isotropic and $C$ is split.
 Hence  $k'$ is a field and $N_C$ is $k'$--isotropic so that $C_{k'}$ is split.
 \end{proof}

\subsection{Cayley-Dickson process}

We know that $C$, up to $k$--isomorphism, can  be obtained by the Cayley-Dickson doubling process, that is
$C \cong C(Q,c)= Q \oplus Q a $ where $Q$ is a $k$-quaternion algebra and $c \in k^\times$ \cite[1.5]{SV}.
We denote by $\sigma_Q= \trd_Q - id_Q$ the canonical involution of $Q$ and
recall that the multiplicativity rule on $C$ (resp. \cskip the norm $N_C$, the canonical involution
$\sigma_C$) is given by
$$
(x+ya)(u+va)= (xu+ c \sigma_Q(v)y ) + ( vx+y \sigma_Q(u)) a  \quad (x,y,u,v \in Q);
$$
$$
N(x+y a)= N(x)- c N(y),
$$
$$
\sigma_C(x+y a)= \sigma_Q(x) - y a.
$$
Then $N_C$ is isometric to the 3-Pfister form $n_Q \otimes \langle 1, -c \rangle$
and that form determines the octonion algebra \cite[1.7.3]{SV}.
Also it provides an embedding $j$  of the  $k$-group $H(Q)= \bigl(\SL_1(Q) \times_k
\SL_1(Q) \bigr)/ \mu_2$ in $\Aut(C(Q,c))$.
This map is given by $(g_1,g_2). (q_1,q_2)= (g_1 \, q_1 \, g_1^{-1} ,g_2 \,q_2 \, g_1^{-1})$.
Another corollary of the determination of an octonion algebra by its norm is the following well-known fact.

\begin{corollary} \label{cor_SV} Let $C$ be a octonion $k$--algebra and let
$Q$ be a quaternion algebra. Then the following are equivalent

\smallskip

(i) There exists $c \in k^\times$ such that $C\cong C(Q, c)$;

\smallskip

(ii) There exists an isometry $(Q,N_Q) \to (C,N_C)$.

\end{corollary}

\begin{proof}
 $(i) \Longrightarrow (ii)$ is obvious. Assume that
 exists an isometry $(Q,N_Q) \to (C,N_C)$. By the linkage propery of Pfister forms \cite[24.1.(1)]{EKM}, there
 exists a bilinear $1$--Pfister form $\phi$ such that $N_C \cong N_Q \otimes \phi$.
 Since $N_C$ represents $1$, we can assume that $\phi$ represents $1$ so
 that $\phi \cong \langle 1,-c \rangle$. Therefore $C$ and $C(Q,c)$ have isometric norms and are isomorphic.
\end{proof}

\begin{remark}{\rm
In odd characteristic, Hooda provided  an alternative  proof, see  \cite[th. 4.3]{H} and also
a nice generalization ({\it ibid}, prop. 4.2).
}
\end{remark}

\begin{lemma}\label{lem_doubling} Let $C$ be a non-split octonion $k$--algebra. If $D\subseteq
C$ is a unital composition subalgebra and $u\in C\setminus D$ then $D \oplus  Du$
is a unital  composition subalgebra as well.
\end{lemma}

\begin{proof} Since $C$ is non-split the corresponding norm map $N_C$
is anisotropic. Let $b_C$ be the polar map of $N_C$. Since the map
$x\mapsto b_C(u,x)$ is linear and the restriction of $b_C$ on $D\times
D$ is regular, there is $v\in D$ such that $b_C(v,x)=b_C(u,x)$
$\forall x\in D$. Let $u'=u-v$. We have $b_C(u',x)=b_C(v,x)-b_C(u,x)=0$
$\forall x\in D$ so $u'\in D^\perp$. Since $v\in D$, $u\notin D$ we
have $u'\neq 0$ so $N_C(u')\neq 0$. By the doubling process \cite[Proposition 1.5.1]{SV}
 we have that $D\oplus D u'$ is a unital composition
subalgebra of $C$. But $u'=u-v$ and $v\in D$ so $D\oplus Du'=D\oplus
Du$.
\end{proof}

\subsection{On the dihedral group, I }\label{sub_diedral}

In this case $W_0= \Aut(\Psi_0)$ and $W_0=D_6= \ZZ/6\ZZ \rtimes \ZZ/2\ZZ= C_2 \times S_3 $
is the dihedral group of order $12$.
More precisely  $C_2=\langle c\rangle$ stands for its center.
The right way to see it is by its action on the root system
 $\Psi(G_0,T_0) \subset \widehat T_ 0= \ZZ \alpha_1 \oplus \ZZ \alpha_2 = \ZZ^2$
 as provided by the following picture

\vskip14mm

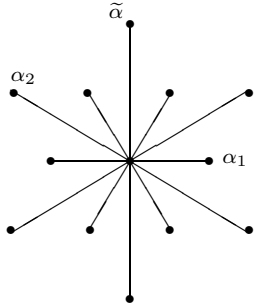
$$
\begin{picture}(0,-0)
\put(0,0){\circle*{3}}
\put(-30,0){\circle*{3}}
\put(0,0){\line(-1,0){30}}
\put(30,0){\circle*{3}}
\put(0,0){\line(1,0){30}}
\put(35,-5){$^{\alpha_1}$}
\put(15,26){\circle*{3}}
\put(0,0){\line(3,5){15}}
\put(-16,26){\circle*{3}}
\put(0,0){\line(-3,5){15}}
\put(15,-26){\circle*{3}}
\put(0,0){\line(3,-5){15}}
\put(-15,-26){\circle*{3}}
\put(0,0){\line(-3,-5){15}}
\put(0,52){\circle*{3}}
\put(0,0){\line(0,1){52}}
\put(-8,50){$^{\widetilde \alpha}$}
\put(0,-52){\circle*{3}}
\put(0,0){\line(0,-1){52}}
\put(45,26){\circle*{3}}
\put(0,0){\line(5,3){45}}
\put(45,-26){\circle*{3}}
\put(0,0){\line(5,-3){45}}
\put(-45,-26){\circle*{3}}
\put(0,0){\line(-5,-3){45}}
\put(-45,26){$^{\alpha_2}$}
\put(-44,26){\circle*{3}}
\put(0,0){\line(-5,3){45}}
\end{picture}
$$

\vskip20mm

\noindent where $\alpha_1, \alpha_2$ stand for a base of the root system $G_2$ and
$\widetilde \alpha= 3 \alpha_1 + 2 \alpha_2$.

Let $\{\varepsilon_i\}_{i=1}^{3}$ be an orthonormal basis  of $\QQ^3$. We can view the root space of $G_2$ as the
hyperplane in $\QQ^3$ defined by
$\Bigl\{\underset{i=1}{\overset{3}{\sum}}\xi_i\varepsilon_i \, | \, \underset{i=1}{\overset{3}{\sum}}\xi_i=0 \Bigr\}$,
and identify $\alpha_1$, $\alpha_2$ with $\varepsilon_1-\varepsilon_2$ and $-2\varepsilon_1+\varepsilon_2+\varepsilon_3$
respectivel (\cite{Bou} Planche IX).
For a root $\alpha$, let $s_\alpha$ be the reflection orthogonal to $\alpha$. Under the above identification, the element
$c=s_{2\alpha_1+\alpha_2}s_{\alpha_2}$ acts on the roots by $-id$ and $S_3=\langle s_{\alpha_1},s_{2\alpha_1+\alpha_2}\rangle$ acts
by permuting the $\varepsilon_i$'s. Note that although $s_{2\alpha_1+\alpha_2}s_{\alpha_2}$ acts on the subspace
$\Bigl\{\underset{i=1}{\overset{3}{\sum}}\xi_i\varepsilon_i \, | \, \underset{i=1}{\overset{3}{\sum}}\xi_i=0 \Bigr\}$ by $-id$,
$s_{2\alpha_1+\alpha_2}s_{\alpha_2}$ does not act as $-id$ on $\{\varepsilon_i\}_{i=1}^{3}$.

Also we  observe that  $\widehat T_0$ is a sublattice of index $2$ of
the lattice $\ZZ \frac{\alpha_1}{2} \oplus \ZZ \frac{\widetilde \alpha}{2}$.
It is related to the fact that the morphism $\SL_2 \times \SL_2 \to G_0$ defined by
the coroots $\alpha_1^\vee$ and   $\tilde{\alpha}^\vee$ has kernel the diagonal
subgroup $\mu_2$.

\begin{remarks}\label{rem_center}
{\rm (a) In the $G_2$ root system, for any long root $\beta$ and any short
root $\alpha$ orthogonal to $\beta$, we have $s_\alpha\circ s_\beta=c$. Also we  observe that  $\widehat T_0$ is a sublattice of index $2$ of
the lattice $\ZZ \frac{\alpha}{2} \oplus \ZZ \frac{\beta}{2}$.
It is related to the fact that the morphism $\SL_2 \times \SL_2 \to G_0$ defined by
the coroots $\alpha^\vee$ and   $\beta^\vee$ has kernel the diagonal
subgroup $\mu_2$.

\noindent (b) The roots $ \alpha_1$, $\widetilde \alpha$ generate a closed symmetric subsystem of type $A_1 \times A_1$ of
$G_2$. Any subroot system  (not necessarily close) of $G_2$ which is of  type $A_1 \times A_1$ is a $W_0$-conjugate  of the previous one.
}
\end{remarks}

\subsection{Subgroups of type $A_1 \times A_1$}
Given an octonion $k$-algebra $C$, 
we relate the Cayley-Dickson's decomposition with subgroups of $G=\Aut(C)$.

\begin{lemma} \label{lem_maximal} Let  $H$ be a semisimple $k$--subgroup
of $G$ of type $A_1 \times A_1$.
 Then there exists a quaternion algebra $Q$, $c  \in k^\times$, an
isomorphism $C \cong C(Q,c)$ and an isomorphism  $H \simlgr H(Q)$ such that the following
diagram commutes
\[
 \xymatrix{
H \enskip \ar[d]^{\wr}  \ar@{^{(}->}[rr] && \enskip G \ar[d]^{\wr} \\
H(Q) \enskip \ar@{^{(}->}[rr]^{j} && \enskip \Aut(  C(Q,c)).
}
\]
\end{lemma}

\begin{proof}
We start with a few  observations  on the split case $G=G_0=\Aut(C_0)$ where we have the $k$--subgroup  $H_0=(\SL_2 \times \SL_2)/\mu_2$
acting on $C_0$. The root subsystem $\Phi(H_0,T_0)$ is $\ZZ {\alpha_1} \oplus \ZZ {\widetilde \alpha}$
so that the first (resp. \cskip the second) factor $\SL_2$ of $H_0$ corresponds to a short (resp. \cskip long) root.
We denote by $H_{0,<} \cong \SL_2$ (resp. \cskip $H_{0,>}$) the ``short'' subgroup (resp. \cskip the ``long'' one)  of $H_0$.
Taking the decomposition $C_0= M_2(k) \oplus M_2(k)_\sharp$, the point is that we have $M_2(k)= (C_0)^{H_{0,>}}$.
In other words, we can recover
the composition subalgebra $M_2(k)$ of $C_0$ from $H_0$.

We come now to our problem. We are given a $k$--subgroup $H$ of $G=\Aut(C)$ of type $A_1 \times A_1$.
Let $T$ be a maximal $k$--torus of $H$.
Then the root system $\Phi(H_{k_s},T_{k_s})$ is a subsystem of $\Phi(G_{k_s},T_{k_s})\cong \Psi_0$ of type $A_1 \times A_1$
hence $W_0$--conjugated to the standard one (Remark \ref{rem_center}.(b)).
Since the Galois action preserves the length of a root, it follows
that we can define by Galois descent the $k$--subgroups $H_{<}$ and $H_{>}$ of $H$.
We define then  $Q= (C)^{H_{>}}$. By Galois descent, it is a quaternion subalgebra of $C$ which is normalized by
$H$. It leads to a Cayley-Dickson decomposition $C = Q \oplus L$ where $L$ is the orthogonal complement of
$Q$ in $C$. Then $L$ is a right $Q$--module and we choose $a \in L$ such that $L=Qa$.
The $k$--subgroup $H(Q)$ of $\Aut(C)$ is nothing but $\Aut(C,Q)$ \cite[\S 2.1]{SV}, so we have $H \subseteq H(Q)$.
For dimension reasons, we conclude that $H=H(Q)$ as desired.
\end{proof}

\section{Embedding a torus in a group of type $G_2$}\label{sec_embed}

We assume that $G$ is a semisimple $k$--group of type $G_2$.
As before in section \ref{sec_image}, we denote by $G_0$ its split form; $T_0$, $W_0$, etc...

\subsection{On the dihedral group, II }
We continue to discuss the action of the dihedral group $W_0$ (of order 12) on the root system
of type $G_2$  started  in \S \ref{sub_diedral}.
Let $\underset{i=1}{\overset{3}{\oplus}}\ZZ\epsilon_i$ be a $W_0$-lattice, where the $S_3$-component
of $W_0$ acts by permuting the
 $\epsilon_i$'s and the center acts by $-id$. Note that $G_0$ is of type $G_2$, so $G_0$ is both adjoint and
simply connected and the dual group of $G_0$ is isomorphic to $G_0$ itself.
Hence we have the following exact sequence of $W_0$-lattices, where $W_0$ acts on $\ZZ$ through its center $\ZZ/2\ZZ$ by $-id$:
\[
\xymatrix@C=0.5cm{
  0 \ar[r] & \widehat T_0 \ar[r]^{f \quad} & \enskip  \underset{i=1}{\overset{3}{\oplus}}\ZZ\epsilon_i \enskip \ar[rr]^{\quad deg} &&\ZZ \ar[r] & 0 },
\]
where $f(\alpha_1)=\epsilon_1-\epsilon_2$ and $f(\alpha_2)=-2\epsilon_1+\epsilon_2+\epsilon_3$.
We also consider its dual sequence:
\[
\xymatrix@C=0.5cm{
  0 \ar[r] & \ZZ \ar[r]&  \underset{i=1}{\overset{3}{\oplus}}\ZZ\epsilon_i^{\vee} \ar[r] & {\widehat T_0}^{\vee}\simeq \widehat T_0 \ar[r] & 0 }.
\]

\subsection{Subtori}\label{sub_tori}
Keep the notations in Section~\ref{sub_diedral}.
Let us fix an isomorphism $$\chi:\ZZ/2\ZZ\times S_3\to \langle c\rangle\times \langle s_{\alpha_1},s_{2\alpha_1+\alpha_2}\rangle=W_0,$$ where
$\chi((-1,1))=c$, $\chi((1,(12)))=s_{\alpha_1}$ and $\chi((1,(23)))=s_{2\alpha_1+\alpha_2}$.

We identify $\ZZ/2\ZZ\times S_3$ with $W_0$ by $\chi$ in the rest of this paper.
Under this identification, we have $$
H^1(k,W_0)= H^1(k,\ZZ/2\ZZ) \times H^1(k,S_3).
$$
Hence a  class of $H^1(k,W_0)$ is  represented uniquely  (up to $k$--isomorphism) by a
couple $(k',l)$ where $k'$ is a quadratic \'etale
algebra of $k$ and $l/k$ is a cubic \'etale algebra of $k$.

Given such a couple  $(k',l)$,  we denote by $\uPsi_{(k',l)}=[(k',l)]\wedge^{W_0}\Psi_0$ the associated
twisted root datum.
Let $l'= l\otimes_k k'$ and define the $k$--torus
$$
T^{(k',l)}=
\Ker\Bigl(  R_{k'/k} \bigl( R^1_{l'/k'}(\GG_{m,l'}) \bigr)
\xrightarrow{ \enskip N_{k'/k} \enskip} R^1_{l/k}(\GG_{m,l})     \Bigr)  .
$$

In the following, we prove that the torus encoded in $\uPsi_{(k',l)}$ is indeed $T^{(k',l)}$.
However, we should keep in mind that two non-isomorphic root data $\uPsi$ may encode the same torus (Remark \ref{rootdata_tori}).

\begin{lemma}\label{lem_shape} Let  $T$ be a $k$--torus of rank two and let $i:T \to G$ be a $k$-embedding
such that $\type(T,i)= \bigl[ (k',l) \bigr]$.

\smallskip

\noindent (1) The $k$--torus $T$ is $k$--isomorphic to
 $T^{(k',l)}$.

 \smallskip

\noindent (2) If there exists a quadratic \'etale algebra $l_2$ such that
$l= k \times l_2$, then
there is a $k$--isomorphism
$$
T \cong \Bigl(R^1_{k_1/k}(\GG_m) \times_k R^1_{k_2/k}(\GG_m) \Bigr) / \mu_2
$$
where $k_1, k_2$ are quadratic \'etale algebras such that $k_2=k'$ and
$[k_1]=[k_2 ] + [l_2] \in H^1(k, \ZZ/2\ZZ)$.
\end{lemma}

\begin{proof} (1) We have $W_0 = \ZZ /2 \ZZ \times S_3$ and from the subsection \ref{sub_diedral}, we have
a $W_0$--resolution
\[
\xymatrix@C=0.5cm{
  0 \ar[r] & \ZZ \ar[r]&  \underset{i=1}{\overset{3}{\oplus}}\ZZ\epsilon_i^{\vee} \ar[r] & \widehat T_0 \ar[r] & 0 }.
\]
It follows that $\widehat T_0$ is isomorphic to the $W_0$-module
 $\underset{i=1}{\overset{3}{\oplus}}\ZZ\epsilon_i^{\vee}/\langle(1,1,1)\rangle$.
Let $N$ be the $W_0$-lattice $\underset{i=1}{\overset{3}{\oplus}}\ZZ e_i/\langle(1,1,1)\rangle$
where $S_3$ acts by permuting the indices and $\ZZ/2\ZZ$ acts trivially.
Note that as $\ZZ$-lattices, we can identify $N$ with $\widehat T_0$.
Let $M=N\oplus N$ and equip $M$ with a $W_0$-action: $S_3$ acts on $N$ diagonally and $\ZZ/2\ZZ$ acts
on $M$ by exchanging the two copies of $N$.
Embed $N$ diagonally into $M$  and we get the following exact sequence of $W_0$-modules:
\[
\xymatrix@C=0.5cm{
  0 \ar[r] & N \ar[rr]^{f \qquad} && M=N\oplus N \ar[rr]^{\qquad g} && \widehat T_0 \ar[r] & 0 },
\]
where $f(x)=(x,x)$ and $g(x,y)=x-y$.  After twisting the above exact sequence  by the $W_0$--torsor attached to $(k',l)$ and
 taking the corresponding tori
we have the following:
\[
\xymatrix@C=0.5cm{
  1 \ar[r] &  T \ar[r] &  R_{k'/k} \bigl( R^1_{l'/k'}(\GG_{m,l'})\bigr) \ar[rr]^{\qquad n_{k'/k}} &&  R^1_{l/k}(\GG_{m,l}) \ar[r] & 1 }.
\]
Hence $T$ is nothing but the $k$--torus $T^{(k',l)}$.

\smallskip

\noindent (2) If $l =k\times l_2$, then there is an injective homomorphism $\iota:\ZZ/2\ZZ\rightarrow S_3$ and a class
$[z]\in\im(\iota_*:H^1(k,\ZZ/2\ZZ)\rightarrow H^1(k,S_3))$ such that $l$ corresponds to $[z]$. Let $\alpha$ be a short root such that the corresponding
reflection $s_\alpha$ is $\iota(-1)$, and let $\beta$ be a long root orthogonal to $\alpha$. As we mentioned in Remark~\ref{rem_center}.(a),  the center of $W_0$
is generated by $s_\alpha\circ s_\beta$. Therefore, the image of the map
$$\Id\times \iota:\ZZ/2\ZZ\times \ZZ/2\ZZ\hookrightarrow\ZZ/2\ZZ\times S_3=W_0$$ is generated by $\{s_\alpha,\ s_\beta\}$. Let us call it
$W^{(k',l_2)}$.
Let $H_0\simeq\bigl(\SL_2 \times_k\SL_2\bigr)/ \mu_2$ be the subgroup of $G_0$ generated by $T_0$ and the root
groups associated to $\pm\alpha$ and $\pm\beta$.
Then $H_0$ is of type $A_1\times A_1$ and the Weyl group of $H_0$ with respect
to $T_0$ is exactly $W^{(k',l_2)}$.
Hence there is $[x]\in\im(H^{1}(k,N_{H_0}(T_0))\rightarrow H^{1}(k,G_0))$ such that $(G,i(T))$ is isomorphic
to $\leftidx_{x}(G_0,T_0)$. Moreover, the embedding $i$ factorizes through $H=\leftidx_{x}(H_0)$.
Let the first (resp. \cskip second) copy of $\SL_2$ of $H_0$ correspond to the root group $\pm\beta$ (resp. \cskip $\pm\alpha$).
Let $\pi$ be the projection from $N_{H_0}(T_0)$ to $N_{H_0}(T_0)/T_0=W^{(k',l_2)}$.
Since $([k'],[l_2])\in H^1(k,\langle s_{\beta}\circ s_{\alpha}\rangle)\times H^1(k,\langle s_\alpha\rangle)=H^1(k,W^{(k',l_2)})$
is equal to
$([k'],[k']+[l_2])\in H^1(k,\langle s_{\beta}\rangle)\times H^1(k,\langle s_\alpha\rangle)=H^1(k,W^{(k',l_2)})$. Under the above
identification,
we have $\pi_*([x])=\bigl([k']+[l_2],[k'] \bigr)\in H^1(k,\langle s_{\alpha}\rangle)\times H^1(k,\langle s_\beta\rangle)$ .
Therefore we have
$$T\simeq \leftidx_{x}(T_0) \cong \Bigl(R^1_{k_1/k}(\GG_m) \times_k R^1_{k_2/k}(\GG_m) \Bigr) / \mu_2,$$ where
$[k_2]=k'$ and $[k_1]=[k_2]+[l_2]$.
\end{proof}

\begin{remark}\label{rootdata_tori} {\rm A natural question is whether the class of $[(k',l)]$ is determined
by the isomorphism class of the torus $T^{(k',l)}$ as $k$--torus. It is not  the case,
there are indeed examples of non equivalent pairs $(k',l)$ and $(k'_\sharp,l_\sharp)$ such that the $k$--tori
$T^{(k',l)}$ and $T^{(k'_\sharp,l_\sharp)}$ are  isomorphic whenever the field $k$ admits a biquadratic field extension $k_1 \otimes_k k_2$.
We put then $k_{1, \sharp}=k_2$ and $k_{2, \sharp}=k_1$.
With the notations of the proof of Lemma \ref{lem_shape}.(2),
we consider the $k$--tori $T=\Bigl(R^1_{k_1/k}(\GG_m) \times_k R^1_{k_2/k}(\GG_m) \Bigr) / \mu_2$ and
$T_\sharp= \Bigl(R^1_{k_{1, \sharp}/k}(\GG_m) \times_k R^1_{k_{2, \sharp}/k}(\GG_m) \Bigr) / \mu_2$.
Then the $k$--tori $T$ and $T_\sharp$ are obviously $k$--isomorphic. However, the root data $\uPsi_{(k',l)}$ and $\uPsi_{(k'_\sharp,l_\sharp)}$ are not isomorphic as $k_2 \not \cong k_{2, \sharp}=k_1$.

Since the pointed set $H^1(k,\GL_2(\ZZ))$ classifies two dimensonal $k$--tori,  the map  $H^1(k,W_0) \to  H^1(k,\GL_2(\ZZ))$ is in this case not injective.
It is due to the fact that the normalizer of
$C_2 \times (1 \times \ZZ/2\ZZ)$ in $\GL_2(\ZZ)$ is larger than the normalizer in $W_0$.

}
\end{remark}

We deal now with the Galois cohomology of those tori.

\begin{lemma}\label{lem_seq}
(1) We have an exact sequence
$$
0 \to \ker\bigl( l^\times \to k^\times \bigr)  / N_{l'/l}\Bigl( \ker\bigl( (l')^\times \xrightarrow{n_{l'/k'}} (k')^\times \bigr) \Bigr)  \, \to \,
H^1(k, T^{(k',l)}) \to
$$
\vskip-4mm
$$
(k')^\times/ N_{l'/k'}\bigl( (l')^\times \bigr)  \xrightarrow{ \enskip n_{k'/k} \enskip}
k^\times/ N_{l/k}(l^\times) \to 0
$$
and the map $n_{k'/k}$ admits a section.
 \smallskip

\noindent (2) Assume that $k'$ and $l$ are fields. Then $H^1\bigl(k,  \widehat{(T^{(k',l)})}^0 \bigr) =0$.
\end{lemma}

\begin{proof} We put $T= T^{(k',l)}$.

\smallskip

\noindent (1)
The Hilbert 90 theorem produces an isomorphism  $k^\times/ N_{l/k}(l^\times)
\simlgr  H^1\bigl(k, R^1_{l/k}(\GG_{m,l}) \bigr)$.
Combined with the Shapiro isomorphism, we get an isomorphism
$(k')^\times/ N_{l'/k'}( {l'}^\times) \simlgr  H^1\bigl(k', R^1_{l'/k'}(\GG_{m,l'}) \bigr)
\simlgr H^1\bigl(k, R_{k'/k} \bigl( R^1_{l'/k'}(\GG_{m,l'}) \bigr)$.
By putting these two facts together,
the  long exact sequence of Galois cohomology reads
$$
\dots \to \ker\bigl( (l')^\times \to (k')^\times \bigr)\xrightarrow{ \enskip N_{l'/l} \enskip}
\ker\bigl( l^\times \to k^\times \bigr) \to
$$
\vskip-4mm
$$
H^1(k,T) \to (k')^\times/ N_{l'/k'}\bigl( (l')^\times \bigr)  \xrightarrow{ \enskip n_{k'/k}
\enskip}
k^\times/ N_{l/k}(l^\times) \, \to \, \dots
$$
Since $k^\times/ N_{l/k}(l^\times) $ is of $3$--torsion, half of the   ``diagonal map''
$k^\times/ N_{l/k}(l^\times)  \to  (k')^\times/ N_{l'/k'}\bigl( (l')^\times \bigr)$ provides
 a section of $(k')^\times/ N_{l'/k'}\bigl( (l')^\times \bigr)  \xrightarrow{ \enskip n_{k'/k} \enskip}
k^\times/ N_{l/k}(l^\times) $.

\smallskip

\noindent (2)
We have an exact sequence
$$
0 \, \to \,  \widehat T^0 \, \to \,  \Coind_k^{k'}( I_{l'/k'} ) \xrightarrow{ \enskip n_{k'/k} \enskip}
 I_{l/k} \, \to \, 0
$$
of Galois modules where $I_{l/k}= \Ker[ \Coind_k^l(\ZZ) \to \ZZ ]$.
It gives rise to the long exact sequence of groups
$$
0 \to H^0( k, \widehat T^0) \, \to \,  H^0\bigl( k,  \Coind_k^{k'}( I_{l'/k'} )\bigr)
\, \to \,  H^0(k,  I_{l/k}) \,  \to \, \dots
$$
\vskip-4mm
$$
H^1( k, \widehat T^0) \to H^1\bigl( k,  \Coind_k^{k'}( I_{l'/k'} )\bigr)
\to H^1(k,  I_{l/k}) \, \to  \, \dots
$$
We consider the exact sequence $0 \, \to \,  I_{l/k}  \, \to \,
\Coind_k^{l}(\ZZ) \,  \to \,  \ZZ \, \to \,  0$ and the corresponding sequence
$$
0 \,  \to \,  H^0(k,I_{l/k}) \, \to \, H^0\bigl(k,\Coind_k^{l}(\ZZ)\bigr) \,\to \,
H^0(k,\ZZ) \to \dots
$$
\vskip-4mm
$$
\to \, H^1(k,I_{l/k}) \, \to \, H^1\bigl(k,\Coind_k^{l}(\ZZ)\bigr) \, \to \, H^1(k,\ZZ).
$$
The group  $\ZZ=H^0\bigl(k,\Coind_k^{l}(\ZZ)\bigr)$ embeds in $\ZZ$ by multiplication by $3$;
also we have  $H^1(k, \Coind_k^l(\ZZ)) \simlgr H^1(l,\ZZ)=0$
by Shapiro's isomorphism. The above sequence induces an isomorphism
$\ZZ / 3 \ZZ \simlgr H^1\bigl(k,I_{l/k} \bigr)$.
On the other hand, we have
$H^1\Bigl(k, \Ind_k^{k'}( I_{l'/k'} )\Bigr)
\, \simlgr \, H^1(k',  I_{l'/k'} ) \simla \ZZ/3 \ZZ$.
The norm map  $n_{k'/k}:  H^1\bigl( \Coind_k^{k'}( I_{l'/k'}) \bigr) \to H^1(k, I_{l/k})$
is the multiplication by $2$ on $\ZZ/3\ZZ$ hence is injective.
By using the starting exact sequence, we conclude that
$H^1\bigl(k,  \widehat{T}^0 \bigr) =0$ as desired.
\end{proof}

\subsection{A  necessary condition}
There is a  basic restriction on the types of maximal tori of $G$.

\begin{proposition}\label{prop_basic}
(1) Let  $T$ be a $k$--torus of rank two and let $i:T \to G$ be a $k$--embedding
such that $\type(T,i)= [(k',l)]$. Then $G\times_k k'$ is split.

\medskip

\noindent (2) Assume that $l= k \times k \times k$. Then the
following are equivalent:

\smallskip

(i) There exists a $k$--embedding $i : T \to G$ of a rank two torus $T$
such that $\type(T,i)= [(k',k^3)]$;

\smallskip

(ii) $G_{k'}$ splits;

  \smallskip

(iii) There is an isometry $(k', n_{k'/k}) \hookrightarrow (C,N_C)$.

\end{proposition}

\begin{proof} (1) Since $G$ is of type $G_2$,
 it is equivalent to show that $G\times_k k'$ is isotropic.

We may assume then $T= T^{(k',l)}$. We consider first the case when $l=k \times l_2$ where $l_2$ is a quadratic \'etale
$k$-algebra. Then we have $T \times_{k} k' \simlgr R^1_{l'/k'}\bigl(\GG_{m,l'}\bigr)
\simla R_{l_2 \otimes k'/k'}(\GG_{m, l_2 \otimes k'})$
hence  $T \times_{k} k'$ is isotropic.

It remains to consider the case when the cubic $k$--algebra $l$ is a field.
From the first case, we see that $G_{l'}$ is split.
In other words the $k'$--group $G_{k'}$ is split by the cubic field algebra $l=l\otimes_k k'$ of $k'$.
Hence  $C_{k'}$ is split hence $C$ splits.

\smallskip

\noindent (2) $(i) \Longrightarrow (ii)$ follows from (1).

\smallskip

\noindent $(ii) \Longrightarrow (i)$.
If $G$ is split, (i) holds according to Theorem \ref{GKR}.
We may assume then $G$ is not split, hence anisotropic.
In particular, $k$ is is infinite.
Since $G_{k'}$ splits,  $k'$ is a field and we denote by $\sigma: k' \to k'$ the conjugacy automorphism.
We use now some classical trick. Since $G(k')$ is Zariski dense in the Weil restriction $R_{k'/k}(G_{k'})$, there
 exists a Borel $k$--subgroup $B$ of $R_{k'/k}(G_{k'})$ such that its conjugate $\sigma(B)$ is opposite to
$B$. The $k$--group $T= B \cap \sigma(B) \cap G$ of $G$ is then a rank two torus.
If we write  $B= R_{k'/k}(B')$ with $B'$ a Borel $k'$-subgroup of $G_{k'}$, then
$T_{k'}$ is a maximal torus of $B'$.
 We denote by $i: T \to G$.
By seeing $i(T_{k'})$ as a maximal $k'$--torus of $B'$, it follows that
the action of $\sigma$ on the root system $\Psi(G_{k'}, T')$ is by $-1$.
Thus $\type(T,i)= (k',k^3)$ as desired.

For the equivalence $(ii) \Longleftrightarrow (iii)$,  see Lemma \ref{lem_quadratic}.
\end{proof}

\begin{remark}   {\rm Another proof of (2) is provided by the next Proposition \ref{prop_biquad}, it is the case
$k_1=k_2$.
}
\end{remark}

\subsection{The biquadratic case}\label{subsec_biquad}

In the dihedral group $D_{6} \subset \GL_2(\ZZ)$, it is convenient to change  coordinates
by considering the diagonal subgroup $(\ZZ/2 \ZZ)^2= \langle c_1 ,\  c_2 \rangle$.
The map $(\ZZ/2 \ZZ)^2 \simlgr C_2 \times (1 \times \ZZ/2\ZZ) \subset C_2 \times S_3$
is given by $(c_1,c_2) \mapsto (c_1, c_1c_2)$.

We are interested in the case when the class of $(k',l)$ belongs to the image
of $H^1( k, \ZZ/ 2\ZZ) \times H^1( k, \ZZ/ 2\ZZ) \to H^1(k, \ZZ/2\ZZ) \times H^1(k,S_3)$.
In terms of \'etale algebras, it rephrases by saying that there are quadratic \'etale $k$-algebras
$k_1/k$, $k_2/k$ such that $k'=k_2$ and $l= k \times l_2$ where $[k_2]= [k_1] + [l_2]$.
We call that case the biquadratic case. In that case, $T^{(k',l)}$ is $k$--isomorphic
to $\Bigl(R^1_{k_1/k}(\GG_m) \times R^1_{k_2/k}(\GG_m)\Bigr)/\mu_2$.

\begin{proposition}\label{prop_biquad} Let $k_1, k_2$ be quadratic \'etale $k$-algebras and denote by
 $\chi_1, \chi_2 \in H^1(k, \ZZ/2\ZZ)$ their class.
We consider  the couple $(k',l)=(k_2,k \times l_2)$
where $[l_2]= [k_1] + [k_2]$. We denote by $\uPsi= \uPsi_{(k',l)}$ defined in Section~\ref{sub_tori}
and by $X=\mathcal{E}(G,\uPsi)$  the $K$-variety of embeddings defined in Section~\ref{subsec_var}.

\smallskip

\noindent (a) The following are equivalent:

\begin{enumerate}
 \item $X(k) \not = \emptyset$, that is $G$ admits a maximal $k$--torus of type $[(k',l)]$;

 \item $C \otimes_k k_j$ is split for $j=1$ and $2$;

 \item  $C$ admits a quaternion subalgebra $Q$  such that there exists $c \in k^\times$ satisfying
 $$[Q] = \chi_1 \cup (c)= \chi_2 \cup (c) \in {_2\!\Br(k)}. $$

\end{enumerate}

\noindent (b) If the $k$--variety $X$ has a zero-cycle of odd degree  then
it has a $k$--point.
\end{proposition}

\begin{proof}
$a)$ If  $C$ is split, the statement is trivial since the three assertions hold.
We can then assume  than $C$ is non split.
We choose scalars $a_1,a_2 \in k$ such that $k_j \cong  k[t]/t^2 -a_j$ for $j=1$ and $2$ if
$k$ is of  odd characteristic and $k_j \cong  k[t]/t^2 +t +a_j$ in the characteristic two case.

\smallskip

\noindent $(1) \Longrightarrow (2)$:
We assume that $T=T^{k',l}\cong \Bigl( R^1_{k_1/k}(\GG_m) \times R^1_{k_2/k}(\GG_m)\Bigr)/ \mu_2 $ embeds then in $G$.
Then $T_{k_j}$ is isotropic so that $G_{k_j}$ is isotropic, hence split  for $j=1$ and $2$.
We conclude that $C_{k_j}$ is split for $j=1$ and $2$.

\smallskip

\noindent $(2) \Longrightarrow (3)$:
We shall construct a quaternion subalgebra  $Q$ of $C$
which contains $k_1$ and $k_2$.
Since $C_{k_j}$ splits for $j=1$ and $2$,
we know that $k_j$ embeds in $C$ as unital composition subalgebra (Lemma \ref{lem_quadratic}).
If $k_1=k_2$ then $Q$ can be obtained from $k_1$
by the doubling process from \cite[Proposition 1.2.3]{SV}.
So we can assume than $k_1\not= k_2$.
Let $x \in k_2 \setminus k_1$. Then Lemma \ref{lem_doubling}
shows that $Q=k_1\oplus k_1 x$ is a unital composition subalgebra
of $C$. It is of dimension $4$ so is a quaternion subalgebra
which contains $k_1$ and $k_2$.
 The common slot lemma yields that there exists $c \in k^\times$ such that
$[Q]=\chi_1 \cup (c)= \chi_2 \cup (c) \in \Br(k)$.
In odd characteristic, a reference for the common slot lemma  is \cite[4.13]{La}.
In characteristic free, it is a consequence of a fact on Pfister forms pointed out by Garibaldi-Petersson \cite[prop. 3.12]{GP}.
The $1$-Pfister quadratic forms  $n_{k_1/k}$ and $n_{k_2/k}$ are subforms of the Pfister quadratic form $N_Q$,
so there exists a bilinear quadratic Pfister form $h= \langle 1, c\rangle$
such that $N_Q \cong  h \otimes n_{k_1/k}  \cong N_Q= h \otimes n_{k_2/k}$.
Thus $[Q]=\chi_1 \cup (c)= \chi_2 \cup (c) \in \Br(k)$ according to the characterization of
quaternion algebras by their norm forms.

\smallskip

\noindent $(3) \Longrightarrow (1)$: We have that $C\cong C(Q,c)$, so we get an embedding \break $\Bigl(\SL_1(Q) \times \SL_1(Q)\Bigr)/ \mu_2
\to \Aut(C(Q,c)) \simlgr G$. By embedding $k_1$ in $Q$ (resp. \ccskip $k_2$ in $Q$), we get
an embedding $ R^1_{k_1/k}(\GG_m) \times  R^1_{k_2/k}(\GG_m) \to \SL_1(Q) \times \SL_1(Q)$ so that
an embedding $i: \Bigl( R^1_{k_1/k}(\GG_m) \times  R^1_{k_2/k}(\GG_m) \Bigr)/ \mu_2 \to \Bigr(  \SL_1(Q) \times \SL_1(Q) \Bigl)/\mu_2
\to G$. By the computations of the proof of Lemma \ref{lem_shape}.(2), it has indeed type $[(k',l)]$.

\smallskip

\noindent (b)
We assume that  $X$ has a $0$--cycle of odd degree, i.e. \cskip there are
finite field extensions  $K_1,\dots, K_r$ of $k$  such that $X(K_i)\not = \emptyset$ for $i=1,\dots,r$
and ${g.c.d.([K_1:K],\dots [K_r:K])}$ is odd.
By (a), it follows that $C_{K_i \otimes_k k_1}$ and $C_{K_i \otimes_k k_2}$ is split for
$i=1,...,r$.
Then there exists an index $i$ such that  $[K_i:k]$ is  odd.
If $k_1=k \times k$, then $C$ splits over $K_i$; it follows  that is $C$  split by Lemma \ref{lem_coprime}
whence $X(k) \not=\emptyset$ by Theorem \ref{GKR}.
We can then assume that $k_1$ is a field.
Then $K_i \otimes_k k_1$ is a field extension of $K_j$ so that $C_{K_j \otimes_k k_1}$
splits; since   $[K_i \otimes_k k_1: k_1]$ is odd; Lemma \ref{lem_coprime} shows then that
$C_{k_1}$ is split.
Similarly $C_{k_2}$ is split and  by (a), we conclude that $X(k) \not= \emptyset$.
\end{proof}

\medskip

In the following, we consider a special case where $k'$ and $l$ have the same discriminant.

\begin{corollary}\label{cor_disc} Let $k'/k$ be a quadratic \'etale algebra and let $l$ be a cubic \'etale $k$--algebra
 of discriminant $k'$. If $C$ admits a maximal $k$--torus of type $[(k',l)]$, then
$C$ splits.
\end{corollary}

\begin{proof} Assume firstly that $l$ is not a field, so that $l \cong  k \times k'$.
Then Proposition \ref{prop_biquad} yields that $C$ is split by the quadratic \'etale $k$--algebra $k_1$
 which satisfies $[k_1]=[k']+[l_2]=0$, whence $C$ is split.

If $l$ is a field, the octonion $l$--algebra $C_l$ admits a maximal $l$-torus of type \break
$[(k' \otimes_k l, l \otimes_k l)]$. Since $l \otimes_k l \simlgr l \times (l \otimes_k k')$, the
first case shows that $C_l$ is split.
We conclude that $C$ is split by Lemma \ref{lem_coprime}.
\end{proof}

\begin{remark}\label{rem_real}{\rm
Take $k = \R$ and let $C$  be  the ``anisotropic'' Cayley algebra (or simply Cayley algebra).
We consider the case where $(k',l)= (\CC, \R \times \C)$. By Corollary \ref{cor_disc}, there is no $\R$--embedding of a maximal torus
of type $(k',l)$. However, $G_{k'}$ splits and this example shows that
only the direct implication holds in Proposition \ref{prop_basic}.(1).
The only possible type is then $[(\CC, \R^3)]$ which is realized according to
Proposition \ref{prop_basic}.(2).
}
\end{remark}

We can now provide a description of such maximal tori.

\begin{proposition}\label{prop_biquad2} Let $k_1, k_2$ be quadratic \'etale $k$-algebras.
We consider  the couple $(k',l)=(k_2,k \times l_2)$
where $[l_2]= [k_1] + [k_2]$ and we assume that $C$ is split by $k_1$ and $k_2$.
We put $T= \bigl( R^1_{k_1/k}(\GG_m) \times R^1_{k_2/k}(\GG_m) \bigr) / \mu_2$ and consider a $k$--embedding
$i: T \to G$ of type $[(k',l)]$. Then there exists a quaternion subalgebra $Q$ of $C$ containing $k_1$ and $k_2$
and a Cayley-Dickson decomposition $C \cong C(Q,c)$ such that $i: T \to G \cong \Aut(C(Q,c))$ factorizes by
the $k$--subgroup $\Bigl(\SL_1(Q) \times \SL_1(Q)\Bigr) / \mu_2$ of $\Aut(C(Q,c))$.
\end{proposition}

\begin{proof}
\noindent{\it Case  $k_1 \otimes_k k_2$ is a field.}
We denote by $\Gamma= \ZZ/2\ZZ \times \ZZ/2\ZZ$ the  Galois group of the biquadratic field extension $k_1 \otimes_k k_2$ .
This group  acts on the root system $\Phi(G_{k_s}, i(T_{k_s}))$ through
a $W_0$--conjugate of the standard  subgroup $\ZZ/2\ZZ \times \ZZ/2\ZZ$ of $W_0$
generated by the central symmetry and the symmetry with the horizontal axis
(see figure in \S \ref{sub_diedral}).
It follows that  $\Gamma$ stabilizes  a subroot system $\Phi_1$ of type $A_1 \times A_1$ of
$\Phi(G_{k_s}, T_{k_s})$.
By Galois descent, the $k_s$-subgroup generated by the root subgroups of $\Phi_1$
descends to a $k$--subgroup $H$ of $G$ which is semisimple
of type $A_1 \times A_1$. Lemma \ref{lem_maximal} shows that
there is a Cayley-Dickson's decomposition $C=Q \oplus Q.a$ such that
$H= H(Q)$. We have then a factorization of $i: T \to G$ by $H(Q) \simlgr \bigl(\SL_1(Q) \times \SL_1(Q)\bigr)/ \mu_2$.

The other cases ($k_1$ or $k_2$ split, $k_1 = k_2$) are simpler of the same flavour
and left to the reader.
\end{proof}

\subsection{The cubic field case: a first example}

Beyond the previous ``equal discriminant case'', the embedding problem for a given octonion algebra $C$
and a couple $(k',l)$ whenever $l$ is a cubic field is much more complicated.
The property to carry  a maximal torus of ``cubic type''
encodes information on the relevant $k$--group and we shall investigate firstly  specific  examples  over Laurent series fields.
 The next fact is inspired by similar considerations on
central simple algebras by Chernousov/Rapinchuk/Rapinchuk \cite[\S 2]{CRR}.

Let us start with a more general setting.
Let $G_0$ be a semisimple Chevalley group defined over $\ZZ$ equipped with a maximal split subtorus $T_0$.
Denote by $\Psi_0$ the root datum attached to $(G_0,T_0)$.
Let $G'/k$ be a quasi-split form of $G_0$ and denote by $T'$
a maximal $k$--torus of $G'$ which is the centralizer of a maximal $k$--split torus of $G'$.
We denote by  $W'=N_{G'}(T')/T'$ the Weyl group of $T'$.


\begin{lemma}\label{lem_residue}
Let $K=k((t))$. Let $E$ be a $W'$-torsor defined over $k$ and put $T=E \wedge^{W'} T'$.
Assume that  $H^{1}(k,\widehat{T}^0)=0$, where $\widehat{T}^0$ is the Galois lattice of cocharacters of $T$.
Let $z : \Gal(K_s/K) \to G'(K_s)$ be a Galois cocycle and put $G={_zG'}/K$.
Assume there is  an embedding $i:T_K\to G$ satisfying  $\type_{can}(i,T_K)=[E]_K \in H^1(K, W')$.
Then $[z]$ is ``unramified'', i.e. \cskip $[z] \in \mathrm{Im}\bigl( H^1(k,G') \to H^1(K,G') \bigr)$.
In particular there exists
a semisimple $k$--group $H$ such that $G \cong H \times_k K$.
\end{lemma}

\begin{proof}
By our form of Steinberg's Theorem \ref{theo_steinberg},
there is a $K$--embedding $i': T_K \to G'_{K}$ such
that the class $[z] \in H^1(K,G'_{K})$ belongs to the image of $i'_*: H^1(K,T) \to H^1(K,G'_{K})$
and furthermore $\type_{can}(T_K,i)= \type_{id}(T_K,i')=[E]_K \in H^1(K,W')$.

On the other hand, we know by Theorem \ref{GKR} that there exists a $k$--embedding $j: T \to G'$ such that $\type_{id}(T,j)=[E]$.
By Proposition \ref{prop_image} the images of $(i')_*$  and  $(j_K)_*: H^1(K,T) \to H^1(K,G')$ coincide.
It follows that  $[z] \in H^1(K,G') $ belongs to  the image of $(j_K)_*: H^1(K,T) \to H^1(K,G')$.
We appeal now to the localization  sequence $ 0 \to H^1(k,T) \to  H^1(K, T) \to H^1(k,\widehat T^0) \to 0$ provided by the appendix (Lemma
\ref{lem_aniso}).  Using our vanishing hypothesis $H^1(k,\widehat T^0)=0$ and the commutative diagram
\[
\xymatrix{
 H^1(k,T) \ar[r]  \ar[d]^{j_{*,k}}
 & H^1(K, T) \ar[d]^{j_{*,K}} \ar[r] & 0\\
 H^1(k,G') \ar[r] & H^1(K, G') &,
 }
\]
we conclude that $[z]$ comes from  $H^1(k,G')$.
\end{proof}

Since every semisimple $K$-group of type $G_2$ is an inner form of its splits form, the  following corollary follows readily.

\begin{corollary}\label{cor_residue}
$K=k((t))$ and let $G/K$ be a semisimple  $K$--group of type $G_2$.
Consider a couple $(k',l)$ such that $k'/k$ is a quadratic \'etale algebra and
$l/k$ is a cubic field separable extension. Denote by $E/k$ the  $W_0$-torsor
associated to $(k',l)$ and put $T/k=E \wedge^{W_0} T_0$.
If the $K$-torus $T \times_k K$ admits an embedding $i$ in $G$ such that  $\type_{can}(T_K,i)=[(k',l)]$, then  there exists
a semisimple $k$--group $H$ of type $G_2$ such that $G \cong H \times_k K$.
\end{corollary}

\begin{proof} We can assume that $G={_z(G_0)}/K$ where $z:\Gal(K_s/K) \to G(K_s)$ is a Galois cocycle.
By Lemma \ref{lem_seq} (2), we have $H^1(k,\widehat{T}^0)=0$. The corollary then follows from Lemma~\ref{lem_residue}
applied to $G'=G_0/k$ and $T'=T_0$.
\end{proof}


\begin{theorem} \label{theo_cycle} Let  $Q$ be a quaternion division algebra over $k$, $k'$ a quadratic \'etale subalgebra of
$Q$ and let $l/k$ be a Galois cubic field extension. As before, let $K=k((t))$, $K'=k'((t))$, $L=l((t))$.
Let $C/K= C(Q_K,t)$ be the octonion algebra algebra built out  from the Cayley-Dickson doubling process.

Let $\uPsi=\uPsi_{(K',L)}$ defined in Section~\ref{sub_tori}, and let $X=\mathcal{E}(G,\uPsi)$ be the $K$-variety of embeddings defined in Section~\ref{subsec_var}.
Then $X(K)= \emptyset$, $X(K') \not = \emptyset$ and $X(L) \not = \emptyset$.
 \end{theorem}

\begin{proof}
We have  $N_C= N_{Q,K} \otimes \langle 1, t \rangle$. Since $N_Q$ is an anisotropic $k$--form,
the quadratic form $N_C$ is anisotropic and cannot be defined over $k$ according to
Springer's decomposition theorem \cite[\S 19]{EKM}.
It follows that the $K$--group $G=\Aut(C)$ cannot be defined over $k$;
Lemma \ref{lem_residue} shows there is no embbeding of a $K$--torus
with type $[(K',L)]$ and therefore $X(K)= \emptyset$.

Since $K'$ splits $C$, $G \times_K K'$ is split so that we have $X(K')\not = \emptyset $  by Theorem \ref{GKR}.
It remains to show that $X(L)$ is not empty.
We have $[(K',L)] \otimes_K L \cong [K' \otimes_K L , L^3]$.
Since $K'$ splits $C$, $K' \otimes_K L $ splits $C$ and
Proposition \ref{prop_basic}.(2) yields that $X(L) \not=\emptyset$.
\end{proof}

\begin{remarks}{\rm
\noindent (a) The requirements on the field $k$ are mild and are satisfied by  any local or global field.

\smallskip

\noindent (b) Geometrically speaking, the variety $X/K$ is a homogeneous space under a $K$--group
of type $G_2$ whose geometric stabilizer is a maximal $K$-torus.
How far as we know, it is the simplest example of homogeneous space under a semisimple
simply connected group with a $0$--cycle of degree
one and no rational points; compare with Florence \cite{Fl} where stabilizers are finite and non-commutative and
Parimala \cite{Pa} where stabilizers are parabolic subgroups.
}
\end{remarks}

\section{\'Etale cubic algebras and hermitian forms }\label{sec_hermitian}

Our goal is to investigate further the cubic case by using
results of Haile-Knus-Rost-Tignol on  hermitian 3-forms \cite{HKRT}.

Let $C$ be an octonion algebra over $k$ and put $G= \Aut(C)$. Let $i: T \to G$ be a $k$--embedding of a rank
$2$ torus and we denote by $[(k',l)]$ its type.

We denote by  $R_{>0}$ the subset of  long roots of the root
system $R=\Phi\bigl(G_{k_s}, i( T_{k_s}) \bigr)$.
Then  $R_{>}$ is a root system of type $A_2$ and is $\Gamma_k$-stable, hence defines
a twisted datum. We consider the $k_s$--subgroup of $G_{k_s}$ generated by $T_{k_s}$ and
the root groups attached to elements of  $R_{>}$, it is a semisimple simply connected
of type $A_2$ and descends to a semisimple $k$--group $J(T,i)$ of $G$.
Our goal is to study such embeddings  $(T,i)$ by means of the subgroup $J(T,i)$.

We shall see in the sequel that such a $k$-group  $J(T,i)$ is a special unitary group form
some  hermitian 3-form for $k'/k$.

\begin{remarks}\label{rems_J} {\rm (a) J.-P. Serre explained us another way to construct the $k$--subgroup $J(T,i)$.
Define the finite $k$--group of multiplicative type
$$\mu_{T,k_s}= \ker\Bigl( T_{k_s} \xrightarrow{\prod \alpha}
\prod\limits_{\alpha \in  R_{>}} \GG_{m,k_s} \Bigl);$$
it descends to a $k$-subgroup $\mu_T$ of $T$.
We claim that $$
J(T,i)= Z_G\bigl( \mu_T \bigr).
$$
For checking that fact, it is harmless to assume that $k$ is algebraically closed.
For simplicity, we put $J= J(T,i)$, it is isomorphic to
$\SL_3$.
Since $\Phi\bigl(J, i( T))\bigr)= R_{>}$, we have
that $\mu_T=Z(J)$ \cite[XIX.1.10.3]{SGA3}; it follows that $\mu_T \cong \mu_3$
and that $J \subseteq  Z_G\bigl( \mu_T \bigr)$. Since $J$ is a semisimple subgroup
of maximal rank of $G$,  the Borel/de Siebenthal's theorem provides
a $k$--subgroup $\mu_n$ of $T$ such that $J=Z_G(\mu_n)$ \cite[Prop. 6.6]{Pe}.
Then $\mu_n \subseteq Z(J) \cong \mu_3$ so that $\mu_n= Z(J)=\mu_T$. Thus
$J=Z_G(\mu_T)$.

\smallskip

\noindent (b) If $k$ is of characteristic $3$, we can associate to $T$ another $k$--subgroup $J_{<}(T,i)$ of type $A_2$.
Let  $R_{<}$ be the subset of short  roots of the root
system $R=\Phi\bigl(G_{k_s}, i( T_{k_s}) \bigr)$. It a $3$-closed symmetric subset \cite[lemma 2.4]{Pe}
so  the $k_s$--subgroup of $G_{k_s}$ generated by $T_{k_s}$ and
the root groups attached to elements of  $R_{<}$ defines a semisimple $k_s$--subgroup $J_{<}$ of $G_{k_s}$ ({\it ibid}, Th. 3.1);
furthermore we have $\Phi\bigl(J, i( T_{k_s}) \bigr)=R_{<}$. The $k_s$--group $J_{<}$ descends to a semisimple $k$--group
 $J_{<}(T,i)$. It is  semisimple of type $A_2$ and adjoint since $R_{<}$ spans $\widehat T(k_s)$.

}
\end{remarks}

\subsection{Rank 3 hermitian forms and octonions}

Let $k'/k$ be a quadratic \'etale algebra.
From a construction of Jacobson \cite[\S 5]{J} (see  \cite[\S 6]{KPS} for the generalization to an arbitrary base field),
we recall that we can attach to a rank $3$ hermitian form $(E,h)$ (for $k'/k)$) with trivial (hermitian) discriminant
an octonion $k$--algebra $C(k',E,h) = k' \oplus E$.
Furthermore, the $k$--group $\SU(k',E,h)$ embeds in $\Aut(C(k',E,h))$ by $g.(x,e) = (x, g.e)$.
We denote by $J(k',E,h)$ this $k$--subgroup and we observe that
$k'$ is the $k$--vector subspace of $C(k',E,h)$ of fixed points for the action  of
$J(k',E,h)$ on $C(k',E,h)$. Also $J(k',E,h)$ is the $k$--subgroup of  $\Aut(C(k',E,h))$ acting trivially
on $k'$.

In  a converse way (see \cite[exercise 6.(b) page 508]{KMRT})), if we are given an embedding of
unital composition $k$--algebra $k' \to C$,
we denote by $E$ the orthogonal subspace  of $k'$ for $N_C$.
For any $x,y \in k'$ and $z \in E$, we have $0=\langle xy,z\rangle_C = \langle y, \sigma_C(x) \, z\rangle_C$ by using the identity
 \cite[1.3.2]{SV}, so that the multiplication $C \times C \to C$ induces a bilinear $k$--map  $k' \times E \to E$.
Then $E$ has a natural $k'$--structure and the restriction of $N_C$ to $E$ defines a hermitian form $h$ (of trivial discriminant)
such that  $C=C(k',E,h)$.

Furthermore, if we have two subfields $k'_1, k'_2$ of $C$ isomorphic to $k'$,
the ``Skolem--Noether'' property \cite[33.21]{KMRT} shows that there exists $g \in G(k)$ mapping $k'_1$ to $k'_2$.
Hence the hermitian forms $(E_1,h_1)$, $(E_2,h_2)$ are isometric.

\begin{remark} {\rm Of course, in such a situation,  $h$ can be diagonalized as $\langle - b, - c, bc \rangle$
and we have  $n_{C(k',E,h)}= n_{k'/k} \otimes \langle \langle b,c  \rangle \rangle$.
If we take $\langle - 1, - 1, 1 \rangle$, we get one form of the split octonion algebra
$C_0$ and then a  $k$--subgroup $J_0=\SL_3$ of $\Aut(C_0)$.
}
\end{remark}

\begin{lemma} In the above setting,  we put $G= \Aut(C(k',E,h))$ and $J= J(k',E,h)$.

\smallskip

\noindent (1) There is a natural exact sequence of algebraic $k$--groups
 $1 \to J \to N_G(J) \to
\ZZ / 2 \ZZ  \to 1$.

\smallskip

\noindent (2) The map $N_G(J)(k) \to
\ZZ / 2 \ZZ$ is onto and      the induced action of $\ZZ/ 2\ZZ$ on $k'$ is the Galois action.

\end{lemma}

\begin{proof} (1) We
consider the commutative exact diagram of $k$--groups
\[
 \xymatrix{
&&& 1 \ar[d] \\
1  \ar[r]   &  Z(J) \ar[d]\ar[r]& J  \ar[d]\ar[r] & J/Z(J) \ar[r] \ar[d]&  1 \\
1  \ar[r]   &  Z_G(J) \ar[r]& N_G(J)  \ar[r] & \Aut(J) \ar[r] \ar[d] &  1 \\
&&& \Autext(J) =\ZZ/2\ZZ \ar[d] \\
&&& 1
}
\]
Let $T$ be a maximal $k$--torus of $J$, it is still maximal in $G$. Then
$Z_G(J) \subseteq Z_G(T) =T$ hence $Z_G(J) \subseteq Z(J)$, so that $Z(J)= Z_G(J)$.
The diagram provides then an exact sequence
$1  \to J \to N_G(J) \to \ZZ/2 \ZZ$.
We postpone the surjectivity.

\smallskip

\noindent (2) Now by the ``Skolem-Noether property'' \cite[33.21]{KMRT}, the Galois action
$\sigma: k' \to k'$ extends to an element  $g \in G(k)$.
Given $u \in J(k)$, $gug^{-1}$ is an element of $G(k)$ which acts
trivially on $k'$, so belongs to $J(k)$. Since it holds for any field extension of
$k$, we have that $g \in N_G(J)(k)$. We conclude that
the map $N_G(J) \to \ZZ/2 \ZZ$ is surjective and that the induced action of
$\ZZ/2 \ZZ$ on $k'$ is the Galois action.
\end{proof}

Let $C$ be an octonion algebra and put $G= \Aut(C)$ and let $J$ be a semisimple  $k$--subgroup of type $A_2$ of $G$.
Then $J$ is of maximal rank and we can appeal again to the Borel-de Siebenthal's classification theorem \cite[th. 3.1]{Pe}.
If the characteristic of $k$ is not $3$, then $J$  is geometrically conjugated to the standard  $\SL_3$ in $G_2$ and is then
simply connected. If the characteristic $k$ is $3$, then $J$ may arise as in Remark \ref{rems_J}.(b) from the short roots
associated to a maximal $k$--torus of $J$; in that case $J$ is adjoint.
We can state  a similar statement as Lemma \ref{lem_maximal}.

\begin{lemma} \label{lem_subgroup}
Let
$J$ be a semisimple simply connected $k$--subgroup of type $A_2$ of $G=\Aut(C)$ and we denote by $k'/k$ the quadratic
\'etale algebra attached to the quasi--split form of $J$.
Then there exists a rank 3 hermitian form $(E,h)$ for $k'/k$ an
isomorphism $C \cong C(k',E,h)$,  an isomorphism  $J \simlgr J(k',E,h)$ such that the following
diagram commutes
\[
 \xymatrix{
J \enskip \ar[d]^{\wr}  \ar@{^{(}->}[rr] && \enskip G \ar[d]^{\wr} \\
J(E,h) \enskip \ar@{^{(}->}[rr]^{j} &&  \enskip \Aut(  C(k',E,h)).
}
\]

\end{lemma}

\begin{proof} Given a $k$-maximal torus $T$ of $G$, we consider the root system $\Psi(G_{k_s},T_{k_s})$. There
are exactly 6 long roots in $\Psi(G_{k_s},T_{k_s})$ which form an $A_2$-subsystem of $\Psi(G_{k_s},T_{k_s})$.
Let $H$ be the subgroup of $G_{k_s}$ which is generated by $T_{k_s}$ and the root groups of long roots. Since the Galois
action preserves the length of a root, the group $H$ is defined over $k$. Hence given a $k$-maximal torus
$T$, there is exactly one subgroup $H$ of $G$ which is a twisted form of $\SL_3$ and contains $T$. Since all maximal
$k$-split tori are conjugated over $k$,
the split group $G_0$ of type $G_2$ has one single conjugacy $G_0(k)$--class of $k$--subgroups isomorphic
to  $\SL_3$.
It follows that the couple $(G,J)$ is isomorphic over $k_s$ to
the couple $(G_0,J_0)$. In particular, by Galois descent, the subspace  of fixed points
of $J$  on $C$ is an \'etale subalgebra $l$ of rank $2$ which is a unital composition subalgebra of $C$.
We define then the orthogonal subspace $E$ of $l$ in $C$.
Then $E$ has a natural structure of $l$--vector space
and carries a hermitian form $h$ of trivial (hermitian) discriminant  such that $C(l,E,h)=C$
(see \cite[exercise 6.(b) page 508]{KMRT}).
But $J$ acts trivially on $l$, so that $J \subseteq  J(l,E,h)$
By dimension reasons, we conclude that  $J =  J(l,E,h)$.
Then $l/k$ is the discriminant \'etale algebra of $J$ hence $k'=l$.
\end{proof}

\begin{remark}{\rm
 Note that in the above proof, we didn't put any assumption on the characteristic of $k$.
 However, in characteristic $\not = 2,3$, Hooda proved the above lemma
in a quite different way \cite[th. 4.4]{H}.
}
\end{remark}

\subsection{Embedding maximal tori}

From now on, we assume for simplicity than the characteristic exponent of $k$ is not $2$.

\begin{lemma}\label{lem_emb} Let $G =\Aut(C)$ be a semisimple $k$--group of type $G_2$.
Let $k'$ (resp. \cskip $l$) be a quadratic (resp. \cskip cubic) \'etale algebra of $k$.
Let $i:T \to G$ be a $k$--embedding  of a  maximal $k$--torus
such  that $\type( T,i) = [(k',l)]$ and denote by $J(T,i)$ the  associated $k$--subgroup of $G$.

\smallskip

(1)  The discriminant algebra of $J(T,i)$ is $k'/k$.

\smallskip

(2) By Lemma \ref{lem_subgroup}, we can write $C= C(k',E,h)$ and identify $J(T,i)$ with $J(k',E,h)$.
Then there  is a $k'$--embedding  $f:k' \otimes_k l \to M_3(k')$ such that
 $f\circ(\sigma \otimes id)=\tau_h\circ f$ on $k' \otimes_k l $, where  $\tau_h$ is the involution on $M_3(k')$ induced by $h$.

\end{lemma}

\begin{proof} (1) We put $J=J(T,i)$. We consider the  Galois action on the root system $\Psi\bigl( G_{k_s}, i(T)_{k_s} \bigr)$
and its subroot system $\Psi\bigl( J_{k_s}, i(T)_{k_s} \bigr)= \Psi\bigl( G_{k_s}, i(T)_{k_s} \bigr)_{>}$.
 It is given by a map  $f: \Gamma_k \to \ZZ/ 2 \ZZ  \times S_3$
defining $[(k',l)]$. Since the Weyl group of  $\Psi\bigl( J_{k_s}, i(T)_{k_s} \bigr)$ is $S_3$, it follows that
the $\star$-action of $\Gamma_k$ on the Dynkin diagram $A_2$ is the projection $\Gamma_k \to \ZZ/ 2\ZZ$.
Therefore   the discriminant algebra of $J(T,i)$ is $k'/k$.

\smallskip

\noindent (2) We have then a $k$--embedding $i: T \to J=\SU(k',E,h)$. Its  type (absolute with respect to $J$) is
$[(k',l)] \in H^1(k, \ZZ/2\ZZ \times S_3)$. By \cite[th 1.15.(2)]{Le}, there is a $k'$--embedding  $k' \otimes_k l \to M_3(k')$ with respect to
the conjugacy involution $\sigma \otimes id$ on $k' \otimes_k l $ and the involution $\tau_h$ attached to $h$.

\end{proof}

\begin{proposition}\label{prop_emb} Let $G =\Aut(C)$ be a semisimple $k$--group of type $G_2$.
 Let $k'$ (resp. $l$) be a quadratic (resp. cubic) \'etale $k$-algebra.
 We denote by $X$ the variety of $k$--embedding of maximal tori in $G$ attached to the twist of $\Psi_0$
by  $(k',l)$ (seen as a $W_0$--torsor). The following are equivalent:

\smallskip

(i) $X(k) \not = \emptyset$, that is there exists an embedding $i:T \to G$ of a maximal $k$--torus of type
$[(k',l)]$.

\smallskip

(ii) there exists a  rank 3 hermitian form $(E,h)$ for $k'/k$ of trivial (hermitian) discriminant
such that $C\cong C(k',E,h)$ and such that there exists a  $k'$-embedding of $k' \otimes_k l \to \End_{k'}(E)$ with respect to
the conjugacy involution on $k'$ and the involution $\tau_h$ attached to $h$.

\smallskip

(iii) there exists a  rank 3 hermitian form $(E,h)$ for $k'/k$ of trivial (hermitian) discriminant
such that $C\cong C(k',E,h)$ and an element $\lambda \in  l^\times$
such that $(l \otimes_k k', \mathrm{t'_\lambda}) \simeq (E,h)$,  where $\mathrm{t'_\lambda}(x,y)=\mathrm{tr}_{l\otimes k'/k'}(\lambda x \, \sigma(y))$.

\end{proposition}

\begin{proof} The implication $(i) \Longrightarrow (ii)$ follows from Lemma \ref{lem_emb}.(2).
Conversely, we assume (ii). Then $G \cong \Aut(C(k',E,h))$ admits the $k$--subgroup  $J(k',E,h) \simlgr \SU(k',E,k)$.
By \cite[th 1.15.(2)]{Le}, there is a $k$--embedding $i: T \to \SU(k',E,k)$ of maximal torus
whose absolute type (with respect to $J$) is $[(k',l)]$.  The $k$--embedding $i: T \to \SU(k',E,k) \to G$
has also absolute type $[(k',l)]$.

The equivalence  $(ii) \Longleftrightarrow (iii)$ follows from the embedding criterion of
$k' \otimes_k l \to \End_{k'}(E)$ given by  proposition 1.4.1 of \cite{BLP}.
\end{proof}

\medskip

Let $k'$, $l$ be as in Proposition \ref{prop_emb}.
Let $\delta\in k^{\times}/k^{\times^{2}}$ be the discriminant of $l$ and $d\in k^{\times}/k^{\times^{2}}$ be the discriminant of $k'$.
Let $B$ be a central simple algebra over $k'$ with an involution $\sigma$ of the second kind. Let $\mathrm{Trd}$ be the reduced trace on $B$.
Let $(B,\sigma)_{+}$ be the $k$-vector space of $\sigma$-symmetric elements of $B$. Let $Q_\sigma$ be the quadratic form on $(B,\sigma)_{+}$
defined by $$Q_{\sigma}(x,y)=\mathrm{Trd}(xy).
$$
Let us recall some  results in  \cite{HKRT}.

\begin{lemma}\label{lem_quad} Assume that $k$ is not of characteristic $2$.
Let $B$ be a central simple $K$-algebra of odd degree $n=2m-1$ with involution $\sigma$ of the second kind. There is a quadratic form $q_{\sigma}$ of
dimension $n(n-1)/2$ and trivial  discriminant over $k$ such that
$$Q_\sigma\simeq \langle 1\rangle\perp\langle 2\rangle\cdot\langle\langle\alpha\rangle\rangle\otimes q_{\sigma}.$$
\end{lemma}

\begin{proof}
We refer to Proposition 4 in \cite{HKRT}.
\end{proof}

\begin{theorem}\label{theo_isometric} Assume that $k$ is not of characteristic $2$ or $3$.
Let $\sigma$, $\tau$ be involutions of the second kind on a central simple algebra $B$ of degree $3$.
 Then $\sigma$ and $\tau$ are isomorphic if and only if
$Q_{\sigma}$ and $Q_{\tau}$ are isometric.
\end{theorem}
\begin{proof}
We refer to Theorem 15 in \cite{HKRT}.
\end{proof}

Let $(B,\sigma)$ be as in Lemma \ref{lem_quad} with degree $B=3$ and assume that $6$ is invertible in $k$.
Let $b_0$, $c_0\in k^{\times}$ such that $q_\sigma\simeq\langle -b_0,-c_0,b_0 c_0 \rangle.$
Define $\pi(B,\sigma)$ to be the Pfister form $\langle\langle d, b_0, c_0\rangle\rangle$.
An involution $\sigma$ of the second kind is called \emph{distinguished} if $\pi(B,\sigma)$ is hyperbolic.
Let $(E,h)$ be a rank 3 hermitian form over $k'$ with trivial (hermitian) discriminant.
We can find $b$, $c\in k^{\times}$ such that $h\simeq\langle -b,-c,bc\rangle_{k'}.$

Now consider the special case where $(B,\sigma)=(\End_{k'}(E),\tau_h)$.
Then we have $q_{\tau_{h}}=\langle -b,-c,bc\rangle$ and $\pi(\End_{k'}(E),\tau_h)=\langle\langle d, b, c\rangle\rangle$,
 which is the norm form of the octonion $C(k',E,h)$. It is then possible to recover with that method at least the two following facts.

\begin{remarks} {\rm
 (a) Theorem \ref{GKR} for $G_2$, i.e. \cskip all possible type of tori occur in the split case.
Given a couple $(k',l)$, we can write the split octonion algebra $C$ as $C(k',E,h)$ for $E=(k')^3$
$h= \langle - 1, -  1, 1\rangle$.
First we note that $l$ can be embedded into $\End_{k'}(E)$ since $\End_{k'}(E)$ is split.
As $N_C$ is isotropic, we have $\tau_h$ is distinguished. By Corollary
 18 in \cite{HKRT}, every cubic \'etale algebra $l$ can
 be embedded as a subalgebra in $\End_{k'}(E)$ with its image in $(\End_{k'}(E),\tau_h)_{+}$. By Proposition \ref{prop_emb}.(2), there is
 an embedding $i:T \to G$ of type $[(k',l)]\in H^1(k, W_0)$.

 \smallskip

\smallskip (b) Corollary \ref{cor_disc}, that is the ``equal discriminant case'':  the discriminant algebra of $l$ is $k'$.
In this case  there is an embedding $i:T \to G$ of type $[(k',l)]$ if and only if $N_C$ is isotropic.
For a proof in the present setting, we
assume there is  an embedding $i:T \to G$ of type $[(k',l)]$. According to Proposition \ref{prop_emb}.(2),
there exists an $3$--hermitian form $(E,h)$  of trivial   determinant
such that $C \cong C(k',E,h)$ and   an embedding $l\underset{k}{\otimes} k'\to \End_{k'}(E)$ with respect
to the conjugacy involution on $k'$ and the involution $\tau_h$ attached to $h$. Then $(\End_{k'}(E),\tau_h)_+$
contains a cubic \'etale algebra isomorphic to $l$ whose discriminant is $d$. By Theorem 16.(e) in \cite{HKRT}, we
have $\pi(\End_{k'}(E),\tau_h)=N_C$ is isotropic. Thus $C$ is split.
}
\end{remarks}

\medskip

\begin{proposition}\label{prop_criterion} Assume that $k$ is not of characteristic $2,3$.
Let $G =\Aut(C)$ be a semisimple $k$--group of type $G_2$.
 Let $k'$ (resp. $l$) be a quadratic (resp. cubic) \'etale $k$-algebra.
 Then there is a $k$-embedding $i:T \to G$ of type $[(k',l)]\in H^1(k, W_0)$ if and only
if the following two conditions both hold:
\smallskip
\item[(i)] There is a rank 3 $k'/k$-hermitian form $(E,h)$ of trivial (hermitian) discriminant such that $C\simeq C(k', E,h)$.
\item[(ii)] Let $b,c\in k^{\times}$ such that $\langle-b,-c,bc\rangle_{k'}$ is isometric to the form $h$ in (i).
 Then there is $\lambda\in l^\times$ such that $N_{l/k}(\lambda)\in k^{\times^{2}}$ and the $k$-quadratic form
 $\langle\langle d \rangle\rangle\otimes \langle\delta\rangle \cdot t_{l/k}(\langle\lambda\rangle)$ is isometric to
 $\langle\langle d\rangle\rangle\otimes \langle -b,-c,bc\rangle$, where $t_{l/k}$ denotes the Scharlau transfer
with respect to the trace map $\mathrm{tr}: l \to k$.
\end{proposition}

\begin{proof}
Suppose that there is a $k$-embedding $i:T \to G$ of type $[(k',l)]\in H^1(k, W_0)$.
By Proposition~\ref{prop_emb} (2), there  is a rank 3 $k'/k$-hermitian form $(E,h)$ such
 that $C\simeq C(k',E,h)$,
and there exists an embedding $\iota: k' \otimes_k l \to \End_{k'}(E)$ with respect to the
 conjugacy involution on $k'$
and the involution $\tau_h$ attached to $h$.
By \cite[Corollary 12]{HKRT}, we can find $\lambda\in l^\times$ such
 that $N_{l/k}(\lambda)\in k^{\times^{2}}$ and the $q_{\tau_{h}}$ in Lemma~\ref{lem_quad} is the
$k$-quadratic form
 $\langle\delta\rangle \cdot t_{l/k}(\langle\lambda\rangle)$.
Since $Q_{\tau_{h}}=3\langle 1\rangle\perp \langle 2\rangle\cdot\langle\langle d\rangle\rangle\otimes \langle -b,-c,bc\rangle$,
the condition $(ii)$ follows from the Witt cancellation.

Conversely, suppose that $(i)$ and $(ii)$ hold.
By Proposition~\ref{prop_emb} (2), it suffices to prove that there is a $k$-embedding of $l$
into $(M_3(k'),\tau_h)_{+}$.
Note that every cubic \'etale $k$--algebra $l$  can be embedded into $M_3(k')$ as $k$-algebras.
By Corollary 14 in \cite{HKRT}, for every $\lambda\in l^{\times}$ such that $N_{l/k}(\lambda)\in k^{\times^{2}}$, there is an involution $\sigma$
 of the second kind  on $M_{3}(k')$ leaving $l$ elementwise invariant such that $$Q_\sigma=
\langle 1,1,1 \rangle\perp\langle 2\rangle\cdot\langle\langle d\rangle\rangle\otimes \langle \delta\rangle\cdot t_{l/k}(\langle\lambda\rangle).$$
Condition $(ii)$ implies that we can choose $\lambda$ so that $Q_\sigma$ and $Q_{\tau_{h}}$ are
 isometric. By Theorem \ref{theo_isometric}, the
 involutions $\sigma$ and $\tau_h$ are isomorphic and hence there is a $k$-embedding of $l$
into $(M_3(k'),\tau_h)_{+}$.
\end{proof}

\section{Hasse principle}\label{sec_hasse}

We assume that the base field $k$ is a number field.

\begin{sproposition}\label{prop_hass} Let $(k',l)$ be a couple where $k'$ is a quadratic \'etale $k$--algebra
and $l/k$  cubic \'etale $k$--algebra. Let $G$ be a semisimple  $k$--group of type
$G_2$ and let $X$ be the $G$--homogeneous space
of the embeddings of maximal tori with respect to the type $[(k',l)]$.
Then $X$ satisfies the Hasse principle.
\end{sproposition}

\begin{proof}
 Since $G_0$ is simply connected, we have $H^1(k_v,G_0)=1$ for each
finite place $v$ of $k$. The Hasse principle states that the map
$$
H^1(k,G_0 ) \simlgr \prod\limits_{v  \enskip \mathrm{real \, \, place}} H^1(k_v, G_0)
$$
is bijective. If $G$ is split, $X(k)$ is not empty (th. \ref{GKR}), so we may assume that $G$ is not split.
By \cite[prop. 2.8]{Le}, $X(k)$ is not empty iff the Borovoi obstruction $\gamma \in
 \Sha^2(k, T^{(k',l)} )$ vanishes.
There is a real place $v$ such that  $G_{k_v}$ is not split and then $k_v$-anisotropic.
Since there is a $k_v$--embedding of $T^{(k',l)}$ in $G_{k_v}$, the torus  $T^{(k',l)}$ is
$k_v$--anisotropic.
By a lemma due to Kneser \cite[lemme 1.9.3]{Sa},  we know that $ \Sha^2(k, T^{(k',l)} )=0$, so that
$\gamma=0$. Thus $X(k) \not = \emptyset$.
\end{proof}

\begin{sremark} {\rm  Under the hypothesis of Proposition \ref{prop_hass}, the existence of a $k$-point on $X$
 is controlled by Borovoi obstruction.
It follows from the restriction-corestriction principle in Galois cohomology that
$X$ has a $k$-point if and only  $X$ has a $0$--cycle of degree one.
In other words, examples like in Theorem \ref{theo_cycle} do not occur over number fields.
}
\end{sremark}

\begin{scorollary}\label{cor_hasse}
Let $k$ be a number field, and $k'$ (resp. $l$)  be  quadratic (resp. cubic) \'etale algebra over $k$.
Let $\delta\in k^{\times}/k^{\times^{2}}$ be the discriminant of $l$ and $d\in k^{\times}/k^{\times^{2}}$ be the discriminant of $k'$.
Let $\Sigma $ be the set of (real) places where $G$ is not split.
Then $T^{(k',l)}$ can be embedded in $G$ with respect to the type $[(k',l)]$ if and only if $d=- 1\in k_v^{\times}/k_v^{\times 2}$
and $\delta=1 \in k_v^{\times}/k_v^{\times 2}$ for each $v \in \Sigma$.
\end{scorollary}

\begin{proof}
According to Proposition \ref{prop_hass},  $T^{(k',l)}$ can be embedded in $G$ with respect to the type $[(k',l)]$ iff
this holds everywhere locally or equivalently (by Theorem \ref{GKR}) iff
this holds  locally on $\Sigma$.
The problem boils down to the real anisotropic case where
the only type is $[(\C, \R^3)]$ according to Remark
\ref{rem_real}.
\end{proof}

\begin{sexamples} \label{ex_eff} {\rm
Keep the notations in Corollary~\ref{cor_hasse}.

\smallskip
\smallskip (a) Consider the special case where $k$ is the field of rational numbers $\QQ$.
Suppose that $G$ is anisotropic over $\QQ$. Since there is only one real place of $\QQ$, by Corollary~\ref{cor_hasse},
the torus $T^{(k',l)}$ can be embedded in $G$ with respect to type $[(k',l)]$ if and only if $k'$ is imaginary and the discriminant of $l$ is positive.

\smallskip
\smallskip (b) Let $k$ be a number field. Suppose that $G$ is anisotropic. Note that in this case, $k$ is a real extension over $\QQ$.
Let $k'$ be an imaginary field extension of $k$ and let the discriminant of $l$ be $[a]\in k^{\times}/k^{\times 2}$ for some positive $a\in\QQ$.
Then by Corollary~\ref{cor_hasse}, the torus $T^{(k',l)}$ can always be embedded in $G$ with respect to type $[(k',l)]$.
}
\end{sexamples}

\section{Appendix: Galois cohomology of tori and semisimple groups over Laurent series fields}\label{sec_tori}

This appendix provides firstly a reference for a well-known fact on the Galois cohomology of
tori  in the vein of the short exact sequence computing the tame Brauer group
of a Laurent series field. This fact is used in the proof of Lemma \ref{lem_residue}.
Secondly we apply our version of Steinberg's theorem to Bruhat-Tits' theory, answering a
question of A. Merkurjev.

We recall that an affine algebraic $k$--group $G$ is a $k$--torus
if there exists a finite Galois extension $k'/k$ such that
$G \times_k k' \simlgr (\GG_{m,k'})^r$.
If $T$ is a $k$--torus, we consider its Galois lattice  of characters
$\widehat T= \Hom_{k_s-gp}\bigl( T_{k_s} ,  \GG_{m,k_s} \bigr)$
and its  Galois lattice  of cocharacters
${\widehat T}^0= \Hom_{k_s-gp}\bigl( \GG_{m,k_s} , T_{k_s}  \bigr)$.

\begin{slemma}\label{lem_aniso} We put $K= k((t))$.
 Let $T/k$ be an algebraic $k$--torus. Then we have a natural split exact sequence
$$
0 \to H^1(k, T) \to H^1(K, T) \buildrel \partial \over  \lgr  H^1\bigl(k,{\widehat T}^0\bigr) \to 0.
$$
\end{slemma}

\begin{proof} Let $k'$ be a Galois extension which splits $T$. We put $\Gamma=\Gal(k'/k)$ and  $K'=k'((t))$.
We have the exact sequence \cite[I.2.6.(b)]{Se}
$$
0 \to H^1(\Gamma, T(k')) \to H^1(k,T ) \to H^1(k',T).
$$
Since $T_{k'}$ is split, the theorem 90 of Hilbert shows that  $H^1(k',T)=0$, whence an isomorphism
$H^1(\Gamma, T(k')) \simlgr  H^1(k,T )$. Similarly, we have $H^1(\Gamma, T(K')) \simlgr  H^1(K,T )$.
We consider the ($\Gamma$--split) exact sequence
$$
1 \to (k'[[t]])^\times \to (K')^\times \to \ZZ \to 0
$$
induced by the valuation. Tensoring with $\widehat T^0$, we get a $\Gamma$-split exact sequence
$$
1 \to T\bigl( k'[[t]] \bigr) \to T(K') \to \widehat T^0 \to 1.
$$
It gives rise to a split exact sequence
$$
0 \to H^1\bigl( \Gamma,  T\bigl( k'[[t]] \bigr) \bigr)
\to H^1\bigl( \Gamma,  T(K') \bigr) \to H^1\bigl(\Gamma ,{\widehat T}^0\bigr) \to 0.
$$
Now we use the  filtration argument of \cite[6.3.1]{GS}
by putting $$U^j = \bigl\{ x \in k'[[t]]^\times \, \mid \, v_t(x-1) \geq j \bigr\}$$ for each $j \geq 0$.
 The $V^j={\widehat T}^0 \otimes  U^j$'s filter  $T\bigl( k'[[t]] \bigr)$
and each $V^{j}/V^{j+1} \cong {\widehat T}^0  \otimes_k k'$ is a $k'$--vector space
equipped with a semi-linear action, hence is $\Gamma$--acyclic\footnote{Speiser's lemma shows that
$V^{j}/V^{j+1}= E_j \otimes_k k'$ for
a $k$-vector space $E_j$
on which $\Gamma$ acts trivially.}
According to the limit fact
\cite[6.3.2]{GS}, we conclude that the specialization map
$H^1\bigl( \Gamma,  T\bigl( k'[[t]] \bigr) \bigr) \to H^1\bigl( \Gamma,  T( k') \bigr)$ is an isomorphism.
We have then a split exact sequence
$$
0 \to H^1\bigl( \Gamma,  T( k') \bigr)
\to H^1\bigl( \Gamma,  T(K') \bigr) \to H^1\bigl(\Gamma,{\widehat T}^0\bigr) \to 0.
$$
Since $H^1\bigl(k',{\widehat T}^0\bigr)=0$, we have
$H^1\bigl(\Gamma ,{\widehat T}^0\bigr) \simlgr  H^1\bigl(k,{\widehat T}^0\bigr)$ whence the desired exact sequence.
\end{proof}

\medskip

Now we relate Bruhat-Tits' theory and our version of Steinberg's theorem \ref{theo_steinberg}.
 Let $G'$ be a quasisplit semisimple $k$--group equipped with a maximal $k$-split subtorus $S'$. We denote by $W'$ the Weyl group
 of the  maximal torus $T'=C_{G'}(T')$ of $G'$.
 Put $K=k((t))$ and denote by $K_{nr}$ the maximal unramified closure of $K$.

\begin{sproposition} \label{prop_BT}
Let $E$ be a $G'_K$--torsor. Then the following are equivalent:

\medskip

(i)  $E(K_{nr}) \not=\emptyset$.

\smallskip

(ii) there exists a $k$--torus embedding $i_0: T_0 \to G'$ such that
 $[E]$ belongs to the image of $i_{0,*}: H^1(K,T_0) \to H^1(K,G)$.
\end{sproposition}

\begin{proof}
We denote by  $G/K={^EG'}$ the inner twist of $G'_K$ by $E$.

\smallskip

\noindent $(i) \Longrightarrow (ii):$
Then $G$ is split by the extension $K_{nr}/K$ and the technical
condition (DE) of Bruhat-Tits theory is satisfied \cite[prop. 5.1.6]{BT}.
It follows that $G$ admits a maximal $K$-torus  $j: T \to G$ which is split
over $K_{nr}$ ({\it ibid}, cor. 5.1.2).

In particular, there exists a $k$--torus $T_0$ such that $T=T_{0,K}$.
We consider now the  oriented type $\gamma= \type_{can}(T',j) \in H^1(K,W')$ provided by the action of
the absolute Galois group of $K$ on the root system $\Phi(G_{K_{s}}, j(T)_{K_s})$.
Since $T$ and $G$ are split by $K_{nr}$, it is given by the action of
$\Gal(K_{nr}/K) \cong \Gal(k_s/k)$  on the root system $\Phi(G_{K_{nr}}, j(T)_{K_{nr}})$
and  then defines a constant class $\gamma_0  \in H^1(k,W')$ such that $\gamma= (\gamma_0)_K$.

In the other hand,  by Kottwitz embedding theorem \ref{GKR}, there exists a $k$--embedding $i_0: T_0 \to G'$ of oriented type
$\gamma_0$. By  theorem \ref{theo_steinberg}, we conclude that
$[E]$ belongs to the image of $i_{0,*}: H^1(K,T_0) \to H^1(K,G')$.

\smallskip

\noindent $(ii) \Longrightarrow (i):$ We assume there is a $k$--embedding $i_0:T_0 \to G'$
such that
$[E]$ belongs to the image of $i_{0,*}: H^1(K,T_0) \to H^1(K,G')$.
Since $T_{0,K}$ is split by $K_{nr}$, the Hilbert 90 theorem shows  that $H^1(K_{nr},T_0)=0$, whence
$E(K_{nr})\not=\emptyset$.
\end{proof}

\begin{remarks}
\noindent (a) If $k$ is perfect, we have that $\mathrm{cd}(K_{nr})=1$ (Lang,
see \cite[th. 6.2.11]{GS}) so condition (i) is always satisfied according to Steinberg's theorem.

\smallskip

\noindent (b) If $k$ is not perfect, there exist examples when condition (i) is not satisfied, even
in the semisimple split simply connected case, see \cite[prop. 3 and Th. 1]{G1}.
\end{remarks}

\bigskip

\medskip

\end{document}